\documentclass{article}
\usepackage{amsfonts,amsmath,amsthm,amssymb,mathrsfs}
\usepackage{fullpage}
\usepackage{enumitem}
\usepackage[latin1]{inputenc}
\usepackage{graphicx,xcolor}
\usepackage{hyperref}
\usepackage{fancyhdr}
\usepackage{lmodern}
\def\ds{\displaystyle}

\def\A{\mathcal{A}}

\newcommand{\ex}{\mathbf{E}}
\newcommand{\PP}{\mathbf{P}}

\newtheorem{theorem}{Theorem}[section]

\newtheorem{remark}[theorem]{Remark}
\newtheorem{lemma}[theorem]{Lemma}
\newtheorem{proposition}[theorem]{Proposition}

\title{Random cubic planar maps}

\author{
	Michael Drmota
	\thanks{
		Institute for Discrete Mathematics and Geometry of the Technische Universit\"at Wien, Austria.
		E-mail: {\tt michael.drmota@tuwien.ac.at}.
		Supported by the Special Research Program SFB F50-02 \textit{Algorithmic and Enumerative Combinatorics}, and by the project P35016  {\it Infinite Singular Systems and Random Discrete Objects} of the Austrian Science Fund FWF.
	}
	\and
	Marc Noy
	\thanks{
		Departament de Matem\`atiques and Institut de Matem\`atiques (IMTech) de la Universitat Polit\`ecnica de Catalunya (UPC), and Centre de Recerca Matem\`atica (CRM), Barcelona, Spain.
		E-mail: {\tt marc.noy@upc.edu}.
		Supported by grants MTM2017-82166-P and PID2020-113082GB-I00, and the Severo Ochoa and Mar\'ia de Maeztu Program for Centers and Units of Excellence in R\&D (CEX2020-001084-M).
	}
	\and 	
	Cl\'ement Requil\'e	
	\thanks{
		Departament de Matem\`atiques and Institut de Matem\`atiques (IMTech) de la Universitat Polit\`ecnica de Catalunya (UPC), Barcelona, Spain.
		E-mail: {\tt clement.requile@upc.edu}.
		supported by the grant Beatriu de Pin\'os BP2019, funded by the H2020 COFUND project No 801370 and AGAUR (the Catalan agency for management of university and research grants), and the grant PID2020-113082GB-I00 of the Spanish ministry of Science and Innovation.
	}
	\and
	Juanjo Ru\'e
	\thanks{
		Departament de Matem\`atiques and Institut de Matem\`atiques (IMTech) de la Universitat Polit\`ecnica de Catalunya (UPC), and Centre de Recerca Matem\`atica (CRM), Barcelona, Spain.
		E-mail: {\tt juan.jose.rue@upc.edu}.
		Supported by grants MTM2017-82166-P and PID2020-113082GB-I00, and the Severo Ochoa and Mar\'ia de Maeztu Program for Centers and Units of Excellence in R\&D (CEX2020-001084-M).
	}
}

\begin{document}

\maketitle

\begin{abstract}
We analyse uniform random cubic rooted planar maps and obtain limiting distributions for several parameters of interest.
From the enumerative point of view, we present a unified approach for the enumeration of several classes of cubic planar maps, which allow us to recover known results in a more general and transparent way.
This approach allows us to obtain new enumerative results.

Concerning random maps, we first obtain the distribution of the degree of the root face, which has an exponential tail as for other classes of random maps. Our main result is a limiting map-Airy distribution law for the size of the largest block $L$, whose expectation is asymptotically $n/\sqrt{3}$ in a random cubic map with $n+2$ faces.
We prove analogous results for the size of the largest cubic block, obtained from $L$ by erasing all vertices of degree two, and for the size of the largest 3-connected component, whose expected values are respectively $n/2$ and $n/4$.
To obtain these results we need to analyse a new type of composition scheme which has not been treated by Banderier et al. [Random Structures Algorithms 2001].
\end{abstract}

%
%
\section{Introduction}\label{sec:intro}

The enumeration of planar maps has a long history, starting with the seminal papers of Tutte in the 1960s \cite{tutteP,tutteT,tutteM}.
Since then, the theory has been extended to maps on arbitrary surfaces and relevant connections have been found between map enumeration and other areas from physics, algebra and probability.

In this work we focus on \textit{cubic} (3-regular) planar maps.
All planar maps considered in this paper will be rooted at a directed edge ${uv}$, where $u$ will be called the \textit{root vertex} and by convention the outer face will be the one to the right of ${uv}$ and will be called the \textit{root face}.
There are numerous papers devoted to the enumeration of cubic maps, and most of them use the direct bijection with triangulations.
The first such result was the enumeration of 3-connected and $4$-connected triangulations by Tutte \cite{tutteT}.
He also counted 2-connected cubic maps \cite{tutteP}, but using a direct approach, while Mullin counted cubic maps using bijections with triangulations \cite{M67}.
More recently, Gao and Wormald \cite{GW02} enumerated simple cubic planar maps, as well as two other classes of cubic maps: simple 2-connected and simple 3-connected triangle-free.
Their proofs are again based on counting the associated triangulations.
The usual approach in the previous references is to consider the more general class of near-triangulations, that are maps in which all faces except possibly the root face have degree three.
Near-triangulations are counted according to the number of faces and the degree of the root face.
Using the quadratic method (see \cite{B65}, and \cite{BMJ06} for a far-reaching generalisation), one is usually able to find an expression for the associated generating function.

\medskip

Our approach is based instead on a direct decomposition of cubic maps, without going through the corresponding dual triangulations.
This approach was already used by Tutte \cite{tutteP} as mentionned above,
and extended by Bodirsky, Kang, L\"offler and McDiarmid \cite{BKLM07} in order to study random cubic planar graphs (see also \cite{NRR20}).
It avoids the quadratic method and, we believe, makes the combinatorial decompositions and the corresponding algebraic formulation more transparent.

As a significant example, it is mentioned by Gao and Wormald in \cite{GW02} that it would be very interesting to find an alternative approach to the enumeration of simple cubic maps.
We provide such an approach which is technically simple and extends the results in \cite{GW02}.
Furhtermore, we recover in a unified way the enumeration of both arbitrary and simple maps with given connectivity.
We are also able to count triangle-free cubic (both arbitrary and simple) maps, a problem considerably more challenging than counting triangle-free 3-connected cubic maps as done in \cite{GW02}.
Using a similar strategy, but technically more involved, the present authors have recently been able to enumerate simple 4-regular maps~\cite{NRR19,NRR21}.

\medskip

In the second part of the paper we obtain limiting distributions of several parameters in a uniform random cubic map.
First, we analyse the degree of the root face.
We show that the probability that the root face has degree $k$ for fixed $k\ge1$ tends to a constant $p_k > 0$.
We show that $\sum_{k\ge1} p_k = 1$ and obtain an explicit (although involved) expression for the probability generation function $p(u) = \sum_{k\ge1} p_k u^k$.
We also deduce the estimate $p_k \sim c\cdot k^{1/2} q^k$ as $k \to \infty$, for computable  constants $c > 0$ and $0 < q < 1$, which conforms to the universal behaviour for tail estimates of the root degree in maps observed in \cite{L99}, and argue why it is expected that the maximum degree is asymptotically $\log_{1/q} n$.

\medskip

Next we analyse the size of the largest \textit{block}, a block being a maximal 2-connected component, and of the largest 3-connected component in random cubic maps.
The first result of this kind was obtained by Bender, Richmond, and Wormald \cite{BRW95}, who showed that the largest 4-connected component in random 3-connected triangulations with $n$ edges has size asymptotically $n/2$.
This was later extended in \cite{GW99} to several types of components in classes of maps, where it was also shown that the size of the second largest component is $O(n^{2/3 + \epsilon})$.
The results from \cite{BRW95} and \cite{GW99} were revisited by Banderier, Flajolet, Schaeffer and Soria \cite{BFSS01}, who showed  that the size of the largest component in many families of random maps obey asymptotically a continuous law of the so-called \textit{map-Airy} type.
This is a particular stable law of index $3/2$, whose relevant properties are recalled in Section \ref{ssec:Airy}.
Let us mention that recently the sizes of blocks in random maps have been analysed using probabilistic tools~\cite{AB19}, reproving part of the results in \cite{BFSS01} and determining for the first time the distribution of the size of $k$-th largest block for~$k\ge2$.

We also study the size of the largest \textit{cubic} block, obtained from the largest block by erasing all vertices of degree two.
To obtain these results we cannot apply directly the techniques developed in \cite{BFSS01}, since the combinatorics of cubic maps differs from the classical families of maps considered so far: when in a cubic map an edge is replaced by a 2-connected map, a new edge in the corresponding block is created and has to be accounted for.

\medskip

We prove a limiting map-Airy distribution for the size of the largest block $L$, whose expectation is asymptotically $n/\sqrt{3}$, where $n$ is the number of edges, then show an analogous result for the size of the largest cubic block, obtained from $L$ by erasing all vertices of degree two, whose expected size is $n/2$.
Finally, we prove the corresponding result for the size of the largest 3-connected component, whose expectation is $n/4$.
It is somehow surprising to obtain these simple constants after a somewhat long analysis involving evaluations of bivariate Cauchy integrals.
We remark that a limiting map-Airy distribution for the size of the largest 3-connected component in random labelled cubic planar \emph{graphs} with $n$ vertices has been determined independently by Albenque, Fusy and Lehéricy \cite{AFL22} and Stufler \cite{S22}, with expected value $\alpha n$, where $\alpha \approx 0.8509$, a key step in proving the scaling limit of random cubic planar graphs.
In the concluding section (Section \ref{sec:conclusions}) we argue why a similar results holds for the size of the largest block in cubic planar graphs.

\medskip

To obtain our results we need to study novel combinatorial schemes, different from the classical scheme $C(wH(z))$ from \cite{BFSS01}, where $C(z)$ is the generating function of cores (defined in Section \ref{ssec:prelim_defs}), $w$ marks the size of the core, and $H(z)$ corresponds to the objects replaced inside the core.
There are several recent papers analysing general composition schemes that go beyond the work in \cite{BFSS01}.
For instance, Banderier, Kuba and Wallner \cite{BKW21} studied schemes of the form $C(wH(z))F(z)$, which generalise the schemes $C(wH(z))$ from \cite{BFSS01}, with applications to trees and lattice paths enumeration.
On the other hand, Stufler \cite{S18,S20,S22Gibbs} analyses so-called Gibbs partitions with applications to random graphs from block-stable classes of graphs.

In our paper we also make a contribution on this line of research: for the largest block the composition scheme is of the form
\begin{equation*}
	B\left( \frac{zw}{1 - zwL(z)} \right)\frac{zw}{1 - zwL(z)},
\end{equation*}
where $w$ marks the size of the 2-core (defined in Section \ref{ssec:prelim_defs}) and $B$ and $L$ are, respectively, the generating functions of 2-connected cubic maps and cubic maps rooted at a loop.
For the largest 3-connected component the scheme is
\begin{equation*}
	M\left( zw(1 + zwD(z)) \right)\frac{1 + 2zwD(z)}{1 + zwD(z)},
\end{equation*}
where now $w$ marks the size of the 3-core and $M$ and $D$ are, respectively, the generating functions of 3-connected cubic maps and cubic maps not rooted at an isthmus.

We are able to analyse these compositions schemes by combining different ingredients.
Our approach is to prove first a limiting map-Airy distribution for the size of the 2- and 3-cores, and then transfer it using a double-counting argument to the size of the largest 2- and 3-connected component.
For the largest block, the analysis is technically demanding as one has to consider simultaneously the size of the largest block and the number of vertices of degree two, combining the fluctuations of an Airy law with those of a Gaussian law.
A similar situation arises for the largest 3-connected component.

\subsection{Results on enumeration}

Our first result (Theorem \eqref{thm:enum}) rediscovers different results that have appeared over the years in the literature by using a unifying method: (a) and (b) were first derived by Mullin et al. \cite{M70} and by Tutte \cite{tutteP}, respectively, while (c) and (d) represent new proofs of Corollaries 3.2 and 4.2 from Gao and Wormald \cite{GW02}, respectively.
As a convention, if $a_n$ denotes a counting sequence of a class of cubic maps, $a^*_n$ denotes the corresponding one for \textit{simple} cubic maps.

\begin{theorem}[\cite{tutteP}, \cite{M70}, \cite{GW02}]\label{thm:enum}
	Let $c_n$ and $b_n$ be respectively the number of arbitrary and 2-connected cubic planar maps with $n+2$ faces.
	Then the following  estimates hold as $n\to\infty$:
	\begin{itemize}
		\item[\rm(a)] $c_n\sim \ds\frac{\sqrt{6}}{\sqrt{\pi}}  n^{-5/2} (12\sqrt{3})^n$,
			where $\ds\frac{\sqrt{6}}{\sqrt{\pi}} \approx 1.38198$ and $12\sqrt{3} \approx 20.78461$.
		\item[\rm(b)] $b_n\sim \ds\frac{\sqrt{3}}{4 \sqrt{\pi}} \cdot n^{-5/2} \left(\frac{27}{2}\right)^n$,
			where $\ds\frac{\sqrt{3}}{4 \sqrt{\pi}} \approx 0.24430$.
		\item[\rm(c)] $c^*_n\sim c \cdot n^{-5/2} \rho^{-n}$,
			where $c \approx 0.16559$ and $\rho^{-1}\approx 10.38845$, where $\rho$ is the smallest positive solution of
			\begin{equation}\label{eq:rho}
				P(z) = 27z^6 + 216z^5 + 171z^4 - 208z^3 - 339z^2 + 24z + 1 = 0.
			\end{equation}
		\item[\rm(d)] $b^*_n\sim b \cdot  n^{-5/2} (5 + 3\sqrt{3})^n$,
			where $b = \ds\frac{(3 + \sqrt{3})\sqrt{2}\sqrt{41\sqrt{3} - 71}}{4\sqrt{\pi}} \approx 0.11201$
			and $5 + 3\sqrt{3}\approx 10.19615$.
	\end{itemize}
\end{theorem}

\paragraph{Remark.}
We notice that two corrections are needed in \cite{GW02}.
In Corollary 3.2, the authors give the values $c\approx 0.0027278757$ in the estimate for $c^*_n$, and in Corollary 4.2 they give the value $b\approx 0.0019155063$ in the estimate for $b^*_n$.

\medskip

Our second result deals with the enumeration of triangle-free cubic maps.
By duality this amounts to count triangulations without vertices of degree three.
This was done in \cite{GW02} for 3-connected maps, but the analogous result for arbitrary maps is more demanding since we have to keep track of triangles in cubic maps.
Indeed they can be created or destroyed when maps are combined in series or in parallel.
To that end, we use decompositions that already proved useful in the graph setting (see \cite{NRR20}).

\begin{theorem}\label{thm:triangle-free}
	Let $f_n$ be the number of triangle-free cubic planar maps with $n+2$ faces.
	Then the following estimates hold as $n\to\infty$:
	\begin{itemize}
		\item[\rm(a)] $f_n\sim f \cdot n^{-5/2} \phi^{-n}$,
			where $f\approx 0.72142$ and $\phi^{-1}\approx 18.18695$,
			and $\phi$ is  a root of
		\begin{equation*}
		\begin{array}{ll}
			22161087866383368192z^{29} - 110805439331916840960z^{28} + 128349633892803674112z^{27} \\
			+ 306063926988657131520z^{26} - 1017316468360256421888z^{25} + 731390086938712080384z^{24} \\
			+ 1412989605840194371584z^{23} - 3904918887380696432640z^{22} + 3286085170959772286976z^{21} \\
			+ 3062041896395210752000z^{20} - 13636190761420628951040z^{19} + 22452065614202935443456z^{18} \\
			- 24015782846601940172800z^{17} + 18890731381294758887424z^{16} - 12618646835081595715584z^{15} \\
			+ 9454042977513918959616z^{14} - 8938299800000420075520z^{13} + 8330326495570886895360z^{12} \\
		 	- 6335783442775792180480z^{11} + 3739491505211342742768z^{10} - 1707114753190595308440z^9 \\
			+ 606877106680714207393z^8 - 169460055073349524800z^7 + 37432243036560849408z^6 \\
			- 6518789166080065536z^5 + 864781240587780096z^4 - 79062401625882624z^3 \\
			+ 3851046019399680z^2 - 14872398004224z - 3131031158784.
		\end{array}
		\end{equation*}

		\item[\rm(b)] $f^*_n\sim f^* \cdot n^{-5/2} (\phi^*)^{-n}$, where $f^*\approx 0.015166$ and $(\phi^*)^{-1}\approx 7.039997$, and $\phi^*$ is a root of
		\begin{equation*}
		\renewcommand{\arraystretch}{1.1}
		\begin{array}{ll}
			& 22161087866383368192z^{29} - 72023535565745946624z^{28} - 217455674688886800384z^{27} \\
			& + 1366192402856046231552z^{26} - 1408884772502960603136z^{25} - 5273526725499791867904z^{24} \\
			& + 18711657605588519485440z^{23} - 20661513660592621092864z^{22} - 15535239133496397004800z^{21} \\
			& + 90959874721137062576128z^{20} - 166070979940102503923712z^{19} + 193400402328142378696704z^{18} \\
			& - 162268637001045608759296z^{17} + 102897252166421987721216z^{16} - 51989933333416282030080z^{15} \\
			& + 24221605189030571544576z^{14} - 13520809952153729316864z^{13} + 9265021383768406435584z^{12} \\
			& - 6064247347538996966656z^{11} + 3267142329643563126000z^{10} - 1396980037043271835032z^9 \\
			& + 473034839943808953505z^8 - 127347508539288938304z^7 + 27332424367753886208z^6 \\
			& - 4657078534989938688z^5 + 614598596098523136z^4 - 58444903901822976z^3 \\
			& + 3329729331462144z^2 - 52444771909632z - 3131031158784.
		\end{array}
		\end{equation*}
	\end{itemize}
\end{theorem}

We can also obtain analogous results for 2-connected triangle-free cubic maps (both arbitrary and simple) but do not include them here to avoid repetition.
The asymptotic estimates are of the same kind and the  growth constants are, respectively, $\psi^{-1}\approx 11.49420$ and $(\psi^*)^{-1}\approx 7.01866$.
For completeness, the number of these maps for small values of $n$ are shown in Section \ref{sec:conclusions}.

\subsection{Results on random cubic maps}

Next, we study properties of the uniform random cubic map $\textsf{M}_n$ with $n+2$ faces as $n\to\infty$.
For some basic  additive parameters, namely the number of cut vertices and isthmuses, we can show convergence to a Gaussian law as $n\to\infty$.
Since the techniques are standard we only display the asymptotic value of the first moment in each case without giving the details:
\begin{equation*}
	\renewcommand{\arraystretch}{1.3}
	\begin{array}{ll}
		\hline
		\text{Parameter} &\quad  \text{Expectation} \\
		\hline
		\text{Cut vertices} &\quad \frac{3}{4} n = 0.75 n \\
		\text{Isthmuses}  &\quad 3\left( 1 - \frac{\sqrt{3}}{2}\right) n\approx 0.40192 \, n \\
		\hline
	\end{array}
\end{equation*}

One can compare these values with the corresponding ones in random cubic planar graphs \cite{NRR20}: for cut vertices the expectation is $\sim 0.00188n$ and for isthmuses it is $\sim 0.00094n$.
The intuition behind this discrepancy in the respective values is that in graphs loops are not allowed, whereas in maps they appear linearly often.
Let us remark that the number of cut vertices is a difficult parameter to analyse in general planar maps \cite{DNS20} but not in cubic maps, since in our case a cut vertex is necessarily incident to an isthmus.
Similarly to \cite{NRR20}, one could also show asymptotic normality for the number of blocks and triangles in $\textsf{M}_n$ but the details would be relatively long and we prefer to concentrate on more novel parameters.

\paragraph{The degree of the root face.}

Our first result is a discrete limit law for the degree of the root face.
Notice that this parameter does not make sense for cubic graphs, since there is no embedding and faces are not defined.
The asymptotic estimate for the tail of the distribution follows the usual form $c\cdot k^{1/2}q^k$ for the degree of the root face, or root vertex in random maps (see \cite{L99}), with $c>0$ and $0 < q < 1$.
More precisely, we have:

\begin{theorem}\label{thm:root_degree}
	For $k\ge 1$, the probability that the root face of $\textsf{M}_n$ has degree $k$ tends to a constant $p_k$ as $n\to\infty$.
	In addition, $\sum_{k\ge 1} p_k = 1$, and
	\begin{equation*}\label{eq:pgf}
		p(u) = \sum_{k\ge1} p_ku^k
	\end{equation*}
	satisfies the cubic equation
	\begin{align*}
		a_3(u)p(u)^3+a_1(u)p(u) +a_0(u)  = 0,
	\end{align*}
	where
	\begin{align*}
		& a_0(u) = 2(211\sqrt{3} - 534 )u^4(4u^6\sqrt{3} + u^7 - 6u^5\sqrt{3} - 9u^6 + 12u^5 - 24u^2\sqrt{3} + 60u\sqrt{3} + 36u^2 - 24\sqrt{3} - 90u + 36), \\
		& a_1(u) = 2(956\sqrt{3} - 1701)u^2(36u^9\sqrt{3} - 2u^{10} - 126u^8\sqrt{3} + 54u^9 + 126u^7\sqrt{3} - 81u^8 - 6u^6\sqrt{3} - 27u^7 - 60u^5\sqrt{3} \\
		& \qquad\qquad\qquad + 54u^6 - 648u^4 + 1944u^3 - 2160u^2 + 864u - 216), \\
		& a_3(u) = 9(12u^4\sqrt{3} + 23u^5 + 18u^3\sqrt{3} + 54u^4 + 24u^2\sqrt{3} + 81u^3 + 66u\sqrt{3} + 108u^2 + 24\sqrt{3} + 90u + 108)(4u^2\sqrt{3} \\
		& \qquad\qquad\qquad + 13u^3 - 26u\sqrt{3} - 36u^2 + 24\sqrt{3} + 78u - 60)^3.
	\end{align*}
	Moreover, the tail of the distribution is of the form
	\begin{equation*}\label{eq:tail}
		p_k \sim c \cdot k^{1/2} q^k, \quad\hbox{ as } k\to\infty,
	\end{equation*}
	where $c \approx 0.032328$, $q \approx 0.90699$ and $q^{-1}$ is the unique positive root of the equation
	\begin{equation*}
			13{u}^{3} + \left(4\sqrt{3} - 36\right){u}^{2} + \left(78 - 26\sqrt{3}\right)u + 24\sqrt{3} - 60.
	\end{equation*}
\end{theorem}
We show below a table containing the first values of $p_k$:
\begin{equation*}
	\renewcommand{\arraystretch}{2.2}
	\begin{array}{c|ccccccc}
		k & 1 & 2 & 3 & 4  & 5 & 6& 7 \\
		\hline
		p_k	& \ds\frac{\sqrt{3}}{36} & \ds\frac{\sqrt{3}}{36} & \ds\frac{\sqrt{3}}{36} & \ds\frac{6\sqrt{3} - 1}{216}
		& \ds\frac{25\sqrt{3}}{864} & \ds\frac{\sqrt{3}}{36} & \ds\frac{35\sqrt{3}}{1296} \\
	\end{array}
\end{equation*}

\paragraph{Largest components.}

Our first result on this topic is a limit law of the map-Airy type for the size of the largest block of $\textsf{M}_n$, whose proof is an adaptation of the method developed in \cite{BFSS01}.
More precisely, if $\A(x)$ is the density function of the Airy distribution (see Subsection \ref{ssec:Airy} for a precise definition) then we have

\begin{theorem}\label{thm:largest_graph_block}
	Let $X_n$ denote the size of the largest block in $\textsf{M}_n$.
	Then, uniformly for $q$ in a bounded interval, we have as $n\to\infty$
	\begin{align*}
		n^{2/3}\PP\left( X_n = \lfloor n/\sqrt{3} + qn^{2/3}\rfloor \right) = c\mathcal{A}(cq)(1 + o(1)),
	\end{align*}
	where $c = 2\sqrt{3}/(1 - 1/\sqrt{3})^{4/3}\approx 10.9215218947$.
\end{theorem}

Notice that a block of a cubic planar map $M$ can have vertices of degree two, hence it is not cubic in general.
This motivates us to define the cubic block associated to a block $B$ of $M$.
It is the 2-connected cubic map obtained from $B$ as follows: for each vertex $v$ of degree two of $B$, with neighbours $a$ and $b$, contract the edge $av$.
Remark that at the end of the process, every vertex of degree two of $B$ has been deleted and we end up with a cubic map (except when $B$ is a cycle).
Our next result is an analogous limit law of the map-Airy type for the size of the largest cubic block of $\textsf{M}_n$. As will be seen in the proofs, the process of counting the removed  vertices of degree two constrains demands a much more refined analysis for cubic blocks than for ordinary blocks.

\begin{theorem}\label{thm:largest_cubic_block}
	Let $X_n^*$ denote the size of the largest cubic block in $\textsf{M}_n$.
	Then, uniformly for $q$ in a bounded interval, we have as $n\to\infty$
	\begin{align*}
		n^{2/3}\PP\left( X_n^* = \lfloor n/2 + qn^{2/3}\rfloor \right) = c^*\mathcal{A}(c^*q)(1 + o(1)),
	\end{align*}
	where $c^* = 4/(1 - 1/\sqrt{3})^{4/3}\approx 12.6110872117$.
\end{theorem}

Finally, we obtain a limit law of the map-Airy type for the size of the largest 3-connected component. Due to the differences between the decomposition of cubic maps into their 3-connected components compared with other families of maps (see the discussion in Section \ref{sec:prelim}), the method from \cite{BFSS01} does not directly apply to this case. But since  3-connected components are always cubic, the scheme developed for the size the largest cubic block can also be used here.

\begin{theorem}\label{thm:largest_3comp}
	Let $Z_n$ denote the size of the largest 3-connected component in $\textsf{M}_n$.
	Then, uniformly for $q$ in a bounded interval, we have as $n\to\infty$
	\begin{align*}
		n^{2/3}\PP\left( Z_n = \lfloor n/4 + qn^{2/3}\rfloor \right) = c'\mathcal{A}(c'q)(1 + o(1)),
	\end{align*}
	where $c' = 72(3/2 - 1/\sqrt{3})^{-4/3}\approx 27.1635288451$.
\end{theorem}

\noindent Let us point that the parameters $c$, $c^*$ and $c'$ quantify in some sense the dispersion of their respective distributions and not the variance since the second moments of $X_n$, $X^*_n$ and $Z_n$ do not exist.

\paragraph{Outline of the paper.}

The rest of the paper is organised as follows.
Section \ref{sec:prelim} contains several preliminary results, in particular the decomposition of cubic maps.
Our main counting results are proved in Section \ref{sec:enum}.
The second part of the paper is devoted to the analysis of parameters in a random cubic map.
In Section \ref{sec:degree} we find the limiting distribution of the degree of the root face.
And in Section \ref{sec:components} we obtain limiting distributions of the map-Airy type for the sizes of the largest block, cubic block and 3-connected component.

%
%
\section{Preliminaries}\label{sec:prelim}

For some background on planar maps we refer the reader to \cite{S15}, and to \cite{D17} for other relevant definitions in graph theory.
We nevertheless explicit some important notions next.

\paragraph{Basic definitions.}

As mentionned in the introduction, all maps considered in this paper are planar and rooted.
A map is \textit{simple} if it has no loops and no multiple edges.
It is \textit{2-connected} if it has at least two vertices, no loops and no cut-vertices, \textit{3-connected} if it has at least four vertices, no 2-cuts and no multiple edges, and \textit{4-connected} if it has at least five vertices and no 3-cuts.
A map is \textit{cubic} if it is 3-regular, and it is a \textit{triangulation} if every face has degree three.
By duality, cubic maps are in bijection with triangulations.
And since duality preserves 3-connectivity, 3-connected cubic maps are in bijection with 3-connected triangulations.
Notice that a general triangulation can have loops and multiple edges, and that a simple triangulation that is not the single triangle is necessarily 3-connected.

At the exception of Section \ref{sec:components}, cubic maps will be counted with respect to the number of \textit{faces minus two.}
By duality, this amounts to counting triangulations by the number of \textit{vertices minus two}.
The smallest cubic maps are the one composed of an isthmus with two loops attached to its endpoints, called the \textit{dumbbell}, and the one composed of two vertices connected by a triple edge, called the \textit{3-bond}; they are depicted by the maps $N_1$ and $N_3$ (respectively) on the left of Figure \ref{fig:replacement}.
As a map, the 3-bond admits a unique rooting, while the dumbbell has three: two on the loops and one on the isthmus.
The smallest triangulations are their respective duals: the triangle which has a single rooting, and the loop with a bridge inside and another one outside.

\subsection{Decompositions of cubic planar maps}\label{ssec:prelim_defs}

\paragraph{3-connected cubic maps.}

Let $T(x)$ and $T_4(z)$ be the generating functions of simple and 4-connected triangulations of the sphere, respectively, where $x$ marks the number of vertices minus two and $z$ the number of faces minus two.
This convention on the variables $x$ and $z$ makes both the algebra and the combinatorics simpler.
A map on $n + 2$ vertices has exactly $3n$ edges and $2n$ faces.
As proven in \cite{tutteT} and \cite{GW02}, the series $T(x)$ and $T_4(z)$ are algebraic functions given by
\begin{equation}\label{eq:Tu}
	T(x) = U(x)\left(1- 2U(x)\right), \qquad x = U(x)(1-U(x))^3,
\end{equation}
and
\begin{equation}\label{eq:T4v}
	T_4(z) = z + V(z)(V(z) - 1)(V(z) + 1)^{-2} - z^2, \qquad z = V(z)(1-V(z))^2.
\end{equation}
By duality the generating function $M(z)$ of 3-connected cubic maps is given by
\begin{equation}\label{eq:3conn_cubic}
	M(z) = T(z) - z,
\end{equation}
where $z$ encodes the 3-bond.

\paragraph{Edge replacement, cherries and beads.}

Given two cubic maps $N$ and $M$, where $st$ is the root edge of $N$, and a directed edge $e=uv$ of $M$, the \textit{replacement} of $e$ by $N$ is the following operation.
Subdivide $e$ twice producing a path $uu'v'v$ in $M$, remove the edge $u'v'$, and identify $u'$ and $v'$ with vertices $s$ and $t$ of $N-st$, respectively.
These results in a cubic map $M'$, whose root edge is that of $M$, unless if $e$ was originally the root of $M$ then there are two possible re-rootings, namely at $uu'$ and $v'v$.
See Figure \ref{fig:replacement} for an illustration.
The reverse operation is called the \textit{removal} of $N$ from $M'$ resulting in the map $M$.

Notice that the replacement of an edge $e$ of $M$ remains valid even when $N$ is rooted at a loop, i.e. $s=t$.
In that case, $N$ is called a \textit{cherry} of $M$ attached at $e$, while when $s\neq t$ it is called a \textit{bead} of $M$ attached at $e$.
On the right of Figure \ref{fig:replacement} is an example of a map $M'$ with one cherry and two beads.
\begin{figure}
\begin{center}
	\begin{minipage}{.48\textwidth}
		\centering
		\includegraphics[scale=1]{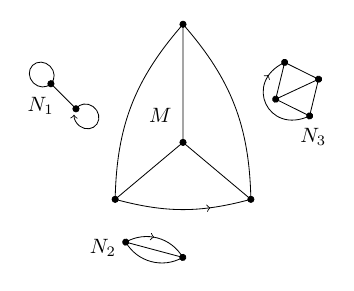}
	\end{minipage}
	\begin{minipage}{.48\textwidth}
		\centering
		\raisebox{.7cm}{\includegraphics[scale=1]{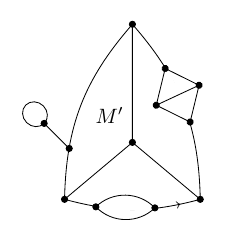}}
	\end{minipage}
\end{center}
\caption{		
	Left: the dumbbell $N_1$ with one of its three possible rootings (one other is at the same loop but in the reverse direction, the third one is at the isthmus), the 3-bond $N_2$ with its unique rooting, and $M$ and $N_3$ which both form the map of the complete graph on four vertices ($K_4$) with its unique rooting.
	Right: the replacement of three edges of $M$ by $N_1$, $N_2$ and $N_3$, resulting in the cubic map $M'$ in which $N_1$ now forms a cherry, while $N_2$ and $N_3$ form beads.
}
\label{fig:replacement}
\end{figure}

\paragraph{2-connected cubic maps.}

Tutte showed in \cite{tutteM} that the family of 2-connected maps can be partitioned into three subclasses, namely series, parallel and polyhedral maps.
Following \cite{BKLM07} (see also \cite{NRR20}), we can easily adapt this idea to the setting of 2-connected cubic maps.
More precisely, let $N$ be a 2-connected cubic map with root $st$.
Then the following three classes form a partition of the class $\mathcal{B}$ of 2-connected cubic maps:
\begin{itemize}
	\item $\mathcal{P}$ (Parallel):
		$N - \{s,t\}$ is not connected.
	\item $\mathcal{S}$ (Series):
		$N - st$ is connected but not 2-connected.
	\item $\mathcal{H}$ (Polyhedral): $N$ is a 3-connected cubic map $C$ where every edge but $st$ is possibly replaced by some map in $\mathcal{B}$.
\end{itemize}	

We let $B(z)$, $P(z)$, $S(z)$ and $H(z)$ be the ordinary generating functions associated to the classes defined above, where once more the variable $z$ {marks the number of faces minus 2}.
This decomposition can be translated into a system of algebraic equations characterizing those generating functions:
\begin{equation}\label{sys:2conn_cubic}
	\def\arraystretch{1.5}
	\begin{array}{lll}
		B(z) &=& S(z) + P(z) + H(z), \\
		S(z) &=& B(z)(B(z) - S(z)), \\
		P(z) &=& z(1 + B(z))^2, \\
		H(z) &=& \ds\frac{M(z(1 + B(z))^3)}{1 + B(z)}.
	\end{array}
\end{equation}
The first equation holds by definition.
The second one follows from the fact that a series map can be decomposed into an arbitrary map in $\mathcal{B}$ and a non-series map in $\mathcal{B-S}$.
The right hand-side of the third equation encodes all possible parallel maps: the $3$-bond ($z$), the maps whose root is in a double edge ($2zB(z)$), and the parallel composition of two maps in $\mathcal{B}$ ($zB(z)^2$).
The last equation corresponds to the definition of the class $\mathcal{H}$, encoded as a composition scheme between the generating functions $M(z)$ and $z(1 + B(z))^3$, i.e. each non-root edge of a 3-connected cubic map is possibly replaced by a map in $\mathcal{B}$.
The cube marks the fact that a cubic map with $n+2$ faces has $3n$ edges, and we divide by $1+B(z)$ to account for the non-replacement of the root edge.

\paragraph{Cubic maps.}

Notice that a cubic map is either 2-connected or it admits an isthmus,
and that a loop is necessarily incident to an isthmus.
We define the following two classes which partition the class of cubic maps that are not 2-connected:
\begin{itemize}
	\item $\mathcal{L}$ (Loop): the root edge is a loop.
	\item $\mathcal{I}$ (Isthmus): the root edge is an isthmus.
\end{itemize}
A map in $\mathcal{L}$ (resp. $\mathcal{I}$) is obtained by possibly replacing by some cubic map the non-root loop (resp. the two loops) of the dumbbell rooted at one of the loops (resp. rooted at an isthmus).

The class $\mathcal{C}$ of cubic maps can then be partitioned as $\mathcal{C} = \mathcal{B} \cup \mathcal{L} \cup \mathcal{I}$.
To  translate this partition into a recursive decomposition, we define the class $\mathcal{D} = \mathcal{L} \cup \mathcal{S} \cup \mathcal{P} \cup \mathcal{H}$ of cubic maps whose root edge is not an isthmus, and we let $L(z)$, $I(z)$ and $D(z)$ be the ordinary generating functions associated to these new classes.
The system \eqref{sys:2conn_cubic} can  be rewritten and extended as (see \cite[Lemma 1]{BKLM07} for a proof):
\begin{equation}\label{sys:cubic}
	\def\arraystretch{1.5}
	\begin{array}{lll}
		C(z) &=& D(z) + I(z), \\
		D(z) &=& L(z) + S(z) + P(z) + H(z), \\
		L(z) &=& 2z(1 + D(z) + I(z)), \\
		I(z) &=& \ds\frac{L(z)^2}{4z}, \\
		S(z) &=& D(z)(D(z) - S(z)), \\
		P(z) &=& z(1 + D(z))^2, \\
		H(z) &=& \ds\frac{M(z(1 + D(z))^3)}{1 + D(z)}.
	\end{array}
\end{equation}
The equations for $P(z)$, $S(z)$ and $H(z)$ are analogue to their counterpart in \eqref{sys:2conn_cubic}, with the difference that they are not restricted to 2-connected cubic maps, in particular edges can be replaced by loop maps.

\paragraph{Cores.}

Let $M$ be a cubic map not rooted at an isthmus, and $C$ be the map obtained from $M$ by iteratively deleting every isthmus while keeping the component containing the root.
If $C$ is 2-connected, then it is called the \textit{2-core} of $M$.
Notice that the 2-core is not in general cubic as it can have vertices of degree two.
The \textit{cubic 2-core} of $M$ is the cubic map obtained after contracting exactly one edge incident to every vertex of degree two of the 2-core of $M$.
Let now $C'$ be the map resulting from the removal of all the cherries and beads of $M$.
If $C$ is cubic and 2-connected, then it is in fact the cubic 2-core of $M$.
If it is furthermore 3-connected, then it is called it the \textit{3-core} of $M$.
Illustrations are given in Figure \ref{fig:cores}.
\begin{figure}
\begin{center}
	\includegraphics[scale=1]{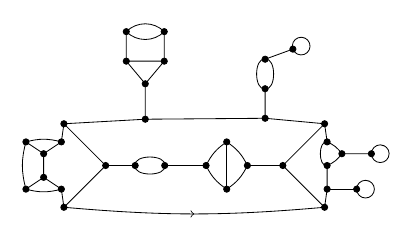}\\
	\includegraphics[scale=1]{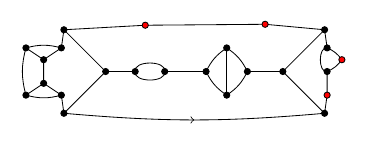}
	\quad
	\includegraphics[scale=1]{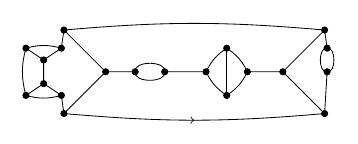}
	\quad
	\includegraphics[scale=1]{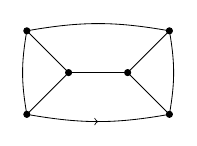}
\end{center}
\caption{
		Top is a cubic map $M$.
		Bottom left is the 2-core of $M$ obtained by contracting each cherry of $M$ to a vertex of degree two (in red).
		Bottom middle is the cubic 2-core of $M$ obtained by contracting one edge incident to each vertex of degree two in the 2-core (or equivalently, obtained directly from $M$ by removing all its cherries).
		Bottom right is the 3-core of $M$ obtained by removing all the beads of $M$.
	}
	\label{fig:cores}
\end{figure}
Note that several cherries and beads attached to the same edge $e$ (as is the case at the top of Figure \ref{fig:cores}) is in fact the result of the replacement of $e$ by a series map.

Cubic maps in the classes $\mathcal{S}$, $\mathcal{P}$ and $\mathcal{H}$ always admit a cubic 2-core, while those in $\mathcal{I}$ and $\mathcal{L}$ never do.
Furthermore, the only cubic maps that admit a 2-core but not a cubic 2-core are those in $\mathcal{S}$ obtained by attaching pending loop maps to the vertices of a rooted cycle (see  case (2) in the proof of Lemma \ref{lem:decomp_2core}).
As pointed out in its definition, every cubic map in the class $\mathcal{H}$ admits a 3-core.
This is also the case for some of the cubic maps in $\mathcal{S}$, namely those obtained by replacing the root edge $uv$ of a map in $\mathcal{H}$ by a map in $\mathcal{D}$; see the series cubic map $M'$ depicted on the right of Figure \ref{fig:replacement} for an example.

\subsection{Asymptotic enumeration}

For $r> 0$, $\varepsilon > 0$, and $0 < \phi < \pi/2$, a \textit{$\Delta$-domain} $\Delta(r,\phi,\varepsilon)$ is a region of the complex plane the form
\begin{equation*}
	\Delta(r,\phi,\varepsilon) = \{ z\in \mathbb{C}: |z| < r + \varepsilon, \ \phi < |{\rm arg}(z-r)|\le \pi \}.
\end{equation*}
The generating function $A(z)$ with non-negative coefficients and radius of convergence $\rho>0$ is said to be $\Delta$-analytic if it admits an analytic continuation around $z=0$ to the domain $\Delta(\rho,\psi,\varepsilon)$.
Furthermore, an algebraic generating function $A(z)$ is said to have a \textit{$3/2$-singularity at $z=\rho$} if for some values $\psi,\varepsilon$ and for $z\sim\rho$ it admits a Puiseux expansion in $\Delta(\rho,\psi,\varepsilon)$ of the form
\begin{equation}\label{eq:3/2sing}
	A(z) = A_0 - A_2\left(1 - \frac{z}{\rho} \right) + A_3\left(1 - \frac{z}{\rho} \right)^{3/2}
	+ O\left( \left(1 - \frac{z}{\rho} \right)^2 \right),
\end{equation}
with $A_0 = A(\rho) > 0$,  $A_1 >0$, $A_2 = \rho A'(\rho) > 0$, and $A_3 > 0$.
In the case where $A(z)$ is an algebraic function, the constants $\rho$, $A_0$, $A_2$ and $A_3$ are algebraic numbers themselves and can be determined, at least implicitly.

The following lemma is an immediate consequence of the Transfer Theorem  (see \cite[Corollary 6.1]{FS09}).

\begin{lemma}\label{lem:technic}
	Let $A(z)$ be a generating function with non-negative coefficients and radius of convergence $\rho>0$.
	Further assume that $A(z)$ has a $3/2$-singularity at $z=\rho$ in the form of \eqref{eq:3/2sing}.
	Then we have
	\begin{equation*}
		[z^n] \, A(z)\sim \frac{3A_3}{4\sqrt{\pi}} \, n^{-5/2} \, \rho^{-n},
		\qquad \hbox{ as } n\to\infty.
	\end{equation*}
\end{lemma}{}

The next lemma is directly adapted from \cite[Theorem 2.31]{D09}.

\begin{lemma}\label{lem:transfer_sing}
Suppose that the generating function $F(x,y)$ has a local expansion of the form
\begin{equation*}
	F(x,y) = g(x,y) + h(x,y)\left(1 - \frac{y}{\rho(x)} \right)^{3/2}
	+ O\left( \left(1 - \frac{y}{\rho(x)} \right)^2 \right),
\end{equation*}
where the function $\rho(x)$ is analytic around $x_0$ such that $\rho(x_0)\ne 0$, and the functions $g(x,y)$ and $h(x,y)$ are analytic around $(x_0,y_0)$ and satisfy $(\partial/\partial y)g(x,y)\ne 1$, $h(x,y)\ne 0$, and $\rho'(x)\ne (\partial/\partial x)g(x,y)$.
Furthermore, assume that $y = y(x)$ is a solution of the functional equation $y = F(x,y)$, with $y(x_0) = y_0$.
Then $y(x)$ has a local expansion of the form
\begin{equation*}
	y(x) = g_1(x) + h_1(x)\left(1 - \frac{x}{x_0} \right)^{3/2}
	+ O\left( \left(1 - \frac{x}{x_0} \right)^2 \right),
\end{equation*}
where $g_1(x)$ and $h_1(x)$ are analytic around $x_0$, and $h_1(x_0)\ne 0$.
\end{lemma}

The proofs of the enumerative results presented in this paper will all follow a common scheme, based on the following steps (we use the terminology and results of~\cite{FS09}, notably Section VII.7.1).

\begin{itemize}
	\item By means of combinatorial decompositions, obtain a system of polynomial equations defining implicitly the generating function of interest $A(z)$.
	Using polynomial elimination, for instance a Gr\"obner basis or successive resultants algorithm, reduce the system to a single bivariate polynomial $P(y,z)$ such that $P(A(z),z) = 0$.
	If $P(y,z)$ is reducible, compute by hand sufficiently many coefficients of $A(z)$ to decide the irreducible factor $Q(y,z)$ of $P(y,z)$ that admits a solution $y(z)$ with the corresponding Taylor expansion at $z=0$.
	As $Q(y,z)$ is irreducible and satisfies $Q(A(z),z) = 0$, it is called the \textit{minimal polynomial} of $A(z)$.
	
	\item Find the dominant singularity $\rho$ of $A(z)$ by looking at the roots of the discriminant of $Q(y,z)$ with respect to $y$.
	By Pringsheim's theorem and due to the fact that $A(z)$ has only non-negative coefficients, $\rho$ will always be a positive real number.
	Prove that $\rho$ is the unique dominant singularity of $A(z)$.
	Since $A(z)$ is algebraic, it is clear that $A(z)$ is then analytic in some $\Delta(\rho,\phi,\varepsilon)$.
	
	\item Using for example the Newton-Puiseux polygon algorithm, compute the Puiseux expansion of $A(z)$ from $Q(y,z)$, in a neighbourhood of $\rho$, corresponding to the branch passing at zero (provided that it holds combinatorially that $A(0) = 0$).
	It will always be of the form of \eqref{eq:3/2sing}.

	\item Conclude with Lemma \ref{lem:technic}.
\end{itemize}

In the rest of the paper when an algebraic generating function $A(z)$ admits a $3/2$-singularity at $z = \rho > 0$, the notation $A_i$ (for $i\ge 0$) will always denote the $i^{\text{th}}$ coefficient of its Puiseux expansion for $z\sim\rho$ in $\Delta(\rho,\psi,\varepsilon)$, and we will omit the mention of the $\Delta$-domain.

\paragraph{An illustrative example.}

As an application of the above scheme, we reprove the estimate on the number of 3-connected cubic planar maps first derived in \cite{tutteT}.
From the equations \eqref{eq:Tu} and \eqref{eq:3conn_cubic} we eliminate $U = U(z)$ and $T = T(z)$ (setting $x = x(z) = z$) to obtain the following irreducible polynomial equation
\begin{equation}\label{poly:3conn_cubic}
	M^4 + (4z + 3)M^3 + (6z^2 + 17z + 3)M^2 + (4z^3 + 25z^2 - 14z + 1)M + z^4 + 11z^3 - z^2 = 0.
\end{equation}
The discriminant with respect to $M$ is $z^2(256z - 27)^3$, whose unique positive root gives the (unique) dominant singularity $27/256$ of $M(z)$.
The Puiseux expansion of $M(z)$ for $z\sim27/256$ is readily computed from \eqref{poly:3conn_cubic} and is equal to
\begin{equation}\label{puis:3conn_cubic}
	 M(z) = M_0 - M_2Z^2 + M_3Z^3 + O(Z^4), \qquad  Z = \sqrt{1 - 256z/27},
\end{equation}
with $M_0 = 5/256$, $M_2 = 21/256$ and $M_3 = \sqrt{6}/24$.
We check that the conditions of Lemma \ref{lem:technic} are satisfied, and we obtain as $n\to\infty$
\begin{equation*}
	[z^n] M(z) = \frac{3M_3}{4\sqrt{\pi}} n^{-5/2} \left(\frac{256}{27}\right)^n (1 + o(1))
	= \frac{\sqrt{6}}{32\sqrt{\pi}} n^{-5/2} \left(\frac{256}{27}\right)^n (1 + o(1)).
\end{equation*}

\subsection{The map-Airy distribution}\label{ssec:Airy}

\paragraph{Density.}

The \textit{map-Airy distribution} (or \textit{Airy distribution of the `map-type'}) has density given by
\begin{equation*}
	\A(x) = 2 e^{-2x^3/3} (x {\rm Ai}(x^2) - {\rm Ai}'(x^2)),
\end{equation*}
where ${\rm Ai}(x)$ is the \textit{Airy function} which satisfies the differential equation $y''-x y = 0$, i.e.
\begin{equation*}
	{\rm Ai}(x) = \frac{1}{2\pi}\int_{-\infty}^{+\infty}  \exp\left( i\left(\frac{t^3}{3} + xt\right) \right)\, dt.
\end{equation*}
The \textit{map-Airy distribution of parameter $c$} is defined by the density $c \A (cy)$.
The tails of the distribution are extremely asymmetric, see a plot of $\A(x)$ shown in Figure~\ref{fig:airy}, in fact the left tail decays polynomially while the right tail decays exponentially:
\begin{align*}\label{eq:airy_integral_representation}
	\A(x) \underset{x\to -\infty}{\sim} \frac{1}{4\sqrt{\pi}}|x|^{-5/2}
		\quad\text{and}\quad
	\A(x) \underset{x\to +\infty}{\sim} \frac{2}{\sqrt{\pi}}x^{1/2} \exp\left( -\frac{4}{3}x^3 \right).
\end{align*}

\begin{figure}
	\centerline{\includegraphics[width=0.7\textwidth]{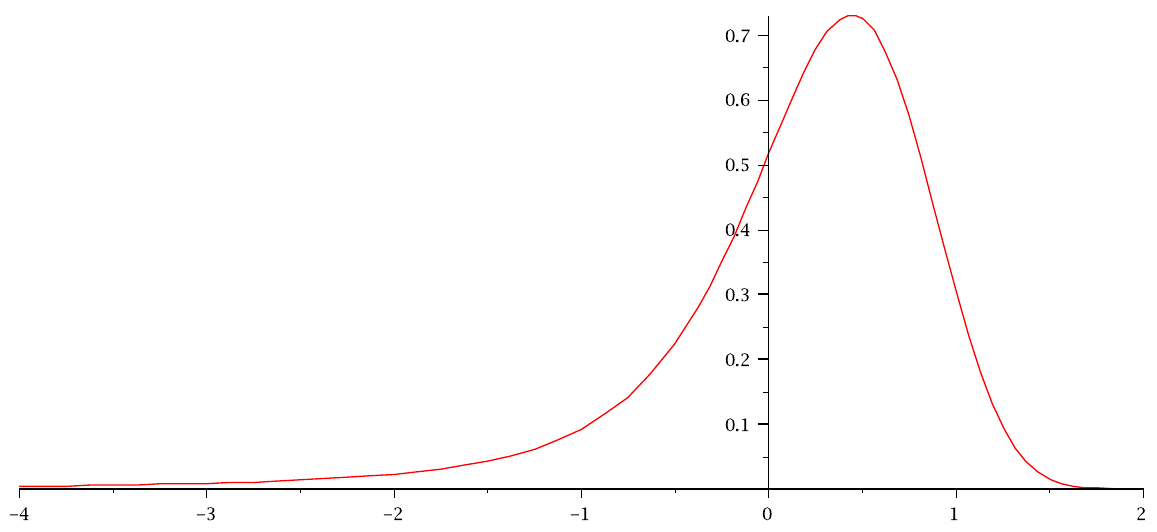}}
	\bigskip \caption{The map-Airy distribution.} \label{fig:airy}
\end{figure}

\paragraph{Integral representations.}

A representation that proves useful in the context of singularity analysis is
\begin{align*}
	\mathcal{A}(x) = \frac{1}{2i\pi} \int_{\infty e^{-i\theta}}^{\infty e^{i\theta}}
	\exp\left(\frac{t^{3/2}}{3} -xt\right)\, dt,
	\qquad \text{for any } \theta\in \left[\frac{\pi}{3}, \frac{2\pi}{3}\right].
\end{align*}
From the above expression we get for any $d_1, d_2 > 0$ that
\begin{equation}\label{eq:airy_integral_representation}
	\frac{1}{2\pi i}\int_{\infty e^{-i2\pi/3}}^{\infty e^{i2\pi/3}}  \exp\left( d_2 s^{3/2} - d_1 ys \right)\, ds
	= (3d_2)^{2/3}\A \left(d_1(3d_2)^{2/3} y\right).
\end{equation}

For an extended account on the Airy distribution as well as other representations (series, etc.), we refer the reader to \cite[Appendix B]{BFSS01}.

%
%
\section{Asymptotic enumeration of cubic planar maps}\label{sec:enum}

\subsection{Cubic maps}

\paragraph{Proof of Theorem \ref{thm:enum}(a).}

Consider the system of equations \eqref{sys:cubic}, where $M$ is defined by \eqref{eq:Tu} and \eqref{eq:3conn_cubic}.
By elimination we obtain the minimal polynomial of $C(z)$ which is equal to
\begin{equation}\label{poly:cubic}
	64C^3z^3 + (192z^3 - 96z^2 + z)C^2 + (192z^3 - 192z^2 + 32z - 1)C + 64z^3 - 96z^2 + 4z.
\end{equation}
Its discriminant with respect to $C$ is $z^2(1-432z^2 )^3$, so that the unique dominant singularity of $C(z)$ is $\sqrt{1/432}=\sqrt{3}/36$.
From \eqref{poly:cubic} we compute the Puiseux expansion of $C(z)$ for $z\sim\sqrt{3}/36$ which is
\begin{equation*}
	C(z) = 6\sqrt{3}-10 - (6\sqrt{3}-12) Z^2 + \frac{4\sqrt{6}}{3} Z^3 + O(Z^4),
	\qquad
	\text{where } Z = \sqrt{1 - 36z/\sqrt{3}}.
\end{equation*}
Applying Lemma \ref{lem:technic} gives the asymptotic estimate for $c_n=[z^n]C(z)$ as claimed. \qed

\paragraph{Proof of Theorem \ref{thm:enum}(b).}

In this case we eliminate from  \eqref{sys:2conn_cubic} and obtain
\begin{equation*}\label{poly:2conn_cubic}
	16z^2B^3 + (48z^2 + 8z)B^2 + (48z^2 - 20z + 1)B + 16z^2 - z.
\end{equation*}
Its discriminant with respect to $B$ is $256z^3(2-27z)^3$.
The unique dominant singularity of $B(z)$ is $2/27$ and the Puiseux expansion for $z\sim 2/27$ is
\begin{equation*}\label{puis:2conn_cubic}
	B(z) = \frac{1}{8} - \frac{3}{8} Z^2 + \frac{\sqrt{3}}{3} Z^3 + O(Z^4),
	\qquad
	\text{where } Z = \sqrt{1 - 27z/2}.
\end{equation*}
The estimate for $b_n = [z^n]B(z)$ follows again from Lemma \ref{lem:technic}. \qed

\begin{remark}
Explicit formulas are known for the coefficients of $C(z)$ and $B(z)$, namely
\begin{equation*}
	c_n = \frac{2^{2n+1}(3n)!!}{(n+2)!n!!},
	\qquad
	b_n = \frac{2^{n+1}(3n)!}{n!(2n+2)!},
\end{equation*}
where $n!! = n(n-2) (n-4)\cdots$ (see \cite{M67} and \cite{tutteP}, respectively).
In those cases, the asymptotic estimates follow by applying Stirling's formula.
However, in the rest of the paper closed formulas will not be available and one needs to rely on methods such as Lemma \ref{lem:technic}.
\end{remark}

\begin{remark}
Furthermore, if we eliminate $S$, $P$, $H$ and $M$ from the system composed of \eqref{eq:3conn_cubic} and \eqref{sys:2conn_cubic}, we obtain the simple identity $B(z) = T(z(1+B)^3)$, which already appears in \cite{tutteP}.
\end{remark}

\paragraph{Proof of Theorem \ref{thm:enum}(c).}

We adapt \eqref{sys:cubic} to encode the generating function $C^*(z) = \sum_{n\ge 0} c^*_nz^n$.
We use the same letters as in \eqref{sys:cubic} for the analogous generating functions but with an exception: we need to redefine the class $\mathcal{D}$ to be the class of cubic maps which become simple after the removal of their root edge.

First we subtract the term $z$ from $L(z)$ corresponding to the dumbbell rooted at a loop, and the term $z$ from $P(z)$ corresponding to 3-bond.
Then we modify the equation for $L(z)$ in order to avoid a double edge when replacing a loop with a loop map.
Finally, $C^*(z)$ is obtained by adding $I(z)$ to $D(z)$ and subtracting the maps giving rise to either loops or double edges, that are the maps encoded in $L(z)$, in $2zD(z)$ (maps obtained from the 3-bond by replacing only one edge), and in $L(z)^2$ (the series composition of two loop maps, which produces a double edge).
This gives
\begin{equation}\label{sys:simple}
	\renewcommand\arraystretch{1.56}
	\begin{array}{lll}
		C^*(z) &=& D(z) + I(z) - L(z) - 2zD(z) - L(z)^2, \\
		D(z) &=& L(z) + S(z) + P(z) + H(z), \\
		L(z) &=& 2z(I(z) + D(z) - L(z)), \\
		I(z) &=& \ds\frac{L(z)^2}{4z},\\
		S(z) &=& D(z)(D(z) - S(z)), \\
		P(z) &=& 2zD(z) + zD(z)^2, \\
		H(z) &=& \ds\frac{M(z(1 + D(z))^3)}{1 + D(z)}.
	\end{array}
\end{equation}

By elimination we obtain the minimal polynomial of $C^*(z)$ as
\begin{equation*}
	64z^5(C^*)^4 + p_3(z)(C^*)^3 + p_2(z)(C^*)^2 + p_1(z)(C^*) + p_0,
\end{equation*}
where
\begin{align*}
	& p_0(z) = z^2(z^2 - 11z + 1)(1568z^8 + 476z^7 - 7456z^6 - 8458z^5 - 27z^4 + 2672z^3 + 130z^2 - 330z + 41), \\
	& p_1(z) = 784z^{11} + 13524z^{10} + 29478z^9 - 51033z^8 - 194686z^7 - 166400z^6 - 5454z^5 + 43746z^4 + 4030z^3 \\
	& \qquad\qquad\qquad - 5652z^2 + 904z - 41, \\
	& p_2(z) = -z(1743z^8 + 13968z^7 + 13344z^6 - 52888z^5 - 116934z^4 - 71248z^3 - 4064z^2 + 3768z - 41), \\
	& p_3(z) = 16z^3(57z^4 + 40z^3 + 24z^2 + 208z + 179).
\end{align*}

The discriminant with respect to $C^*$ has two real roots between 0 and 1.
One of them is approximately $0.32$, and can be discarded since it is greater than the dominant singularity of the generating function of 2-connected simple cubic maps, approximately $0.099$, computed in the next proof.
The second one is $\rho\approx 0.096260$, a root of $P(z)$ defined in \eqref{eq:rho}, as claimed.
The Puiseux expansion of $C^*(z)$ for $z\sim\rho$ is
\begin{equation*}
	 C^*(z) = C^*_0 - C^*_2 Z^2 +  C^*_3Z^3 + O(Z^4), \qquad
	 \text{where } Z = \sqrt{1 - z/\rho},
\end{equation*}
and $C^*_0$, $C^*_2$ and $C^*_3$ are computable polynomials in $\rho$.
Their expressions are too long to be reproduced here, and we just write down the numerical approximations:
\begin{equation*}
	 C^*_0\approx 0.020004, \qquad
	 C^*_2\approx 0.14836, \qquad
 	 C^*_3\approx 0.39135.
\end{equation*}
The asymptotic estimate on $c^*_n $ follows again from Lemma \ref{lem:technic}. \qed

\paragraph{Proof of Theorem \ref{thm:enum}(d).}

We restrict the system \eqref{sys:simple} to \textit{2-connected} simple maps.
To this end we need to discard the classes of simple maps that produce cut vertices, namely $\mathcal{L}$ and $\mathcal{I}$.
The generating functions $D(z)$, $S(z)$, $P(z)$ and $H(z)$ have the same meaning as for simple cubic maps, except that they are now restricted to 2-connected simple cubic maps.

Similarly to the previous proof, $B^*(z)$ is obtained from $D(z)$ by removing the maps containing a double edge, which are parallel maps encoded by $zD(z)^2$.
\begin{equation*}\label{sys:2conn_simple}
	\renewcommand\arraystretch{1.56}
	\begin{array}{lll}
		B^*(z) &=& D(z) - 2zD(z), \\
		D(z) &=& S(z) + P(z) + H(z), \\
		S(z) &=& D(z)(D(z) - S(z)), \\
		P(z) &=& 2zD(z) + zD(z)^2, \\
		H(z) &=& \ds\frac{M(z(1 + D(z))^3)}{1 + D(z)}.
	\end{array}
\end{equation*}
By elimination we obtain the  minimal polynomial of $B^*(z)$ as follows
\begin{align*}
	16z^2(B^*)^3 - (16z^4 + 120z^3 - 48z^2 - 8z)(B^*)^{2}
	& + (4z^6 + 76z^5 + 121z^4 - 244z^3 + 118z^2 - 20z + 1)B^* \\
	& - 8z^7 - 76z^6 + 134z^5 - 77z^4 + 17z^3 - z^2.
\end{align*}
The smallest positive root of the discriminant with respect to $B^*$ is $(3\sqrt{3}-5)/2$.
And the Puiseux expansion for $z\sim (3\sqrt{3}-5)/2$ is given by
\begin{equation*}
	B^*(z) = \frac{33\sqrt{3} - 57}{8} - \frac{25 - 15\sqrt{3}}{8} Z^2
	+ \frac{(3 + \sqrt{3})\sqrt{2}\sqrt{41\sqrt{3} - 71}}{3} Z^3 + O(Z^4),
	\ \text{where } Z = \sqrt{1 - z(5 + 3\sqrt{3})}.
\end{equation*}
Once more, we obtain the asymptotic estimate for $b^*_n$ from Lemma \ref{lem:technic}.\qed

\subsection{Triangle-free cubic  maps}

We need to adapt \eqref{sys:cubic} to encode the decomposition of triangle-free cubic maps.
The first obstacle arises from the edge-replacement operation: when replacing an edge of a map $M$ by a map $N$ we can potentially erase or create triangles, and the resulting map becomes or stop being triangle-free.

To address this problem, we will encode whether the root edge of a map belongs or not to a triangle.
We will also keep track of triangular faces in 3-connected cubic maps.
The latter is in order to control whether at least one of the edges of each of those faces gets replaced with a map.
By duality, this is equivalent to keeping track of cubic vertices in 3-connected triangulations.

\paragraph{Cubic vertices in 3-connected triangulations.}

We introduce two bivariate generating functions: $T_1(x,u)$ which counts 3-connected triangulation with at least five vertices and rooted at a cubic vertex, and $T_0(x,u)$ which counts those that are not rooted at a cubic vertex.
In both cases, $x$ marks the number of vertices minus two and $u$ marks the number of cubic vertices.

In order to derive functional equations for $T_0$ and $T_1$, we will proceed as in \cite[Section 4.1]{NRR20} (to which we refer the reader for a more complete proof): start from the univariate generating function $T_4(z)$ of 4-connected triangulations (see \eqref{eq:T4v}), since 4-connected triangulations cannot have cubic vertices, and introduce an auxiliary generating function $T^{(3)}(x,u)$ which counts 3-connected triangulations and where now $u$ marks the number of \textit{inner} cubic vertices, that is, not in the outer face.
The triangulation $K_4$ is treated separately.
This gives the following system of equations:
\begin{equation}\label{eq:T1T0T3}
	\renewcommand\arraystretch{1.56}
	\begin{array}{lll}
		T^{(3)}(x,u) &=& \ds\frac{T_4(x(1 + x^{-1}T^{(3)}(x,u))^2)}{1+ x^{-1}T^{(3)}(x,u)} + x^2(1 + x^{-1}T^{(3)}(x,u))^3 + x^2(u - 1), \\
    	T_1(x,u) &=& uxT^{(3)}(x,u), \\
    	T_0(x,u) &=& (1+ 2xu - 3x)T^{(3)}(x,u) - x^2u.
    \end{array}
\end{equation}


\paragraph{Proof of Theorem \ref{thm:triangle-free}(a).}

We use the same letters as in  \eqref{sys:cubic} for the generating functions, with the difference that those with the index 1 (resp. 0) will encode cubic maps that are triangle-free with the exception of at least one triangle (resp. no triangle) incident with the root edge.
They will be called \textit{near-triangle-free} (resp. triangle free).

The series $S_1(z)$, $P_1(z)$ and $H_1$ (resp. $S_0(z)$, $P_0(z)$ and $H_0$) encode near-triangle-free (resp. triangle-free) cubic maps thats are respectively series, parallel and polyhedral.
The maps counted by $L(z)$ and $I(z)$ never have triangles at the root, so we omit their index.
Additionally, the polyhedral maps with $K_4$ as a 3-core will be encoded by $W_0(z)$ and $W_1(z)$.
We now prove the next lemma, which is an adaptation of \cite[Lemma 21]{NRR20} to the setting of planar maps.

\begin{lemma}\label{lem:triangle-free}
	The generating function $F(z) = \sum_{n\ge 0}f_nz^n$, where $f_n$ ($n\ge 0$) is the number of triangle-free cubic maps on $n+2$ faces, satisfies the  following system of equations:
	\begin{equation}\label{sys:triangle-free_cubic}
    	\renewcommand\arraystretch{1.3}
		\begin{array}{lll}
			F &=& I + L + S_0 + P_0 + W_0 + H_0, \\
        	D &=& L + S_0 + P_0 + W_0 + H_0 + S_1 + P_1 + W_1 + H_1, \\
        	L &=& 2z(1 + I + D - L^2 - z) - 4z^2(D - L), \\
        	I &=& \ds\frac{L^2}{4z}, \\
        	P_0 &=& z(1 + D - L)^2, \\
        	P_1 &=& zDL, \\
        	S_0 &=& (D - S_0 - S_1)D - S_1, \\
        	S_1 &=& 2zL + 4z(D - L)L + L^3, \\
        	W_0 &=& z^2(4D^2 + 8D^3 + 5D^4 + D^5), \\
        	W_1 &=& z^2(D + 6D^2 + 2D^3), \\
            H_1 &=& \ds\frac{T_1(x,u)}{3D + 3D^2 + D^3}, \\
        	H_0 &=& (2D + D^2)H_1 + \ds\frac{T_0(x,u)}{1 + D}, \\
    	\end{array}
	\end{equation}
	where the arguments of the univariate series are omitted, $x = z(1 + D)^3$, and $u = (3D + 3D^2 + D^3)/(1 + D)^3$.
\end{lemma}

\begin{proof}
	The first two equations follow from their respective definitions.
	The rest of the proof goes into the lines of \cite[Lemma 21]{NRR20}: setting $x^2=z$ and $u=0$ in \cite[Lemma 17]{NRR20}.
	In order to adapt it to the setting of maps, we remove the various graph symmetries (encoded by the factor $1/2$ in \cite[Lemma 21]{NRR20}), add a factor 2 for the choice of the root face, and add/remove the terms with loops or multiple edges.
	For instance, the terms $2z$ and $-2z^2$ in the equation for $L(z)$, the latter creating a triangle not at the root edge coming from the added term $z$ in the equation for $P_0(z)$.
	And the terms $2zL(z)$ and $z^2D(z)$ in the equations for $S_1(z)$ and $W_1(z)$, respectively.
	Notice finally that all the generating functions with index 2 in \cite[Lemma 21]{NRR20} have been fused here with those of index 1.
\end{proof}

By elimination from the system composed of \eqref{eq:T4v}, \eqref{eq:T1T0T3} and \eqref{sys:triangle-free_cubic}, and setting $x = z(1 + D)^3$ and $u = (3D + 3D^2 + D^3)/(1 + D)^3$, we obtain an irreducible polynomial equation $p_F(F,z)=0$ which has degree 24 in~$F$, and is too large to be reproduced here.
The discriminant with respect to $F$ has four factors with positive roots smaller than one.
Only one such factor has a positive root $\phi\approx 0.054984$ larger than $\rho$, the dominant singularity of all cubic maps.
Hence $\phi$ must be the dominant singularity of $F$.
Its defining equation is given in the statement of Theorem \ref{thm:triangle-free}(a).
From $p_F$ we compute the Puiseux expansion of $F(z)$ for $z\sim\phi$:
\begin{equation*}
	F(z) = F_0 - F_2Z^2 + F_3Z^3 + O(Z^4),
	\qquad
	\text{where } Z = \sqrt{1 - z/\phi}.
\end{equation*}
The $F_i$'s are computable algebraic functions of $\phi$ that are too large to be displayed here, and we only give numerical approximations
\begin{equation*}
	F_0\approx 0.35300, \quad F_2\approx 1.05162, \quad F_3\approx 1.70491.
\end{equation*}
We conclude the proof by applying Lemma \ref{lem:technic}. \qed

\medskip

Since the next proof is very similar to the proof of the previous theorem, we only provide a sketch of it.

\paragraph{Sketch of the proof of Theorem \ref{thm:triangle-free}(b).}

We  adapt the system \eqref{sys:triangle-free_cubic} to the case of simple triangle-free cubic maps encoded by $F^* = F^*(z)$.
This is done by taking care of possible appearances of loops or multiple edges that are not the root.
This gives the following system where the arguments of the funtions are omitted:
\begin{equation}\label{sys:triangle-free_simple}
    \renewcommand\arraystretch{1.3}
	\begin{array}{lll}
    	F^* &=& I + D - S_1 - P_1 - W_1 - H_1 - L^2 - 2z(D - L), \\
        L &=& 2z(I + D - 2z(D - L)  - L^2), \\
        P_0 &=& 2z(D - L) + z(D - L)^2, \\
        S_1 &=& 4z(D - L)L + L^3.
    \end{array}
\end{equation}

From there, we compute the minimal polynomial $p_{F^*}(F^*,z)$ which also has degree 24 in $F^*$.
By carefully analysing its discriminant, we obtain the dominant singularity $\phi^*\approx 0.142046$, whose minimal polynomial $P_{\phi^*}$  is given in the statement of Theorem  \ref{thm:triangle-free}(a).
Finally we compute the Puiseux expansion as before and obtain the estimate for $[z^n]F^*(z)$. \qed

%
%

\section{The degree of the root face}\label{sec:degree}

We continue with the convention that $z$ marks faces minus two in a cubic map.
Let $C(z,u)$ be the generating function of cubic maps, where $u$ marks the degree of the root face, and let ${M}(z,u)$ be the analogous series for 3-connected cubic maps.
We first find an expression for $M(z,u)$ using the fact that the number of loopless maps with $n$ edges and root face of degree $k$ equals  the number of 3-connected cubic maps with $n+2$ faces and root face degree $k+2$; see also a bijective proof of this fact in \cite{F10}.

The generating function $A(x,y)$ of loopless maps, where $x$ marks edges and $y$ marks the root face degree was obtained in \cite{BW85} as follows.
The univariate function $A_1(x) = A(x,1)$ is given by
\begin{equation}\label{eq:loopless_univ}
	A_1(z) = A(z,1) = (1 + q)^2(1 - q),
\end{equation}
where $q = q(x)$ satisfies $q = x(1 + q)^4$, which in fact is the generating function of 4-ary trees.
We notice that  $A_1(x) = x^{-1} T(x)$, where $T(x)$ is as in Section \ref{sec:prelim}, and that the unique singularity of $q$ agrees with that of $T$, namely $x = \tau = 27/256$.
It is further shown in \cite{BW85} that $A = A(x,y)$ is the solution of
\begin{equation}\label{eq:loopless}
	xyA^2 + (1 - y - xyA_1)A + y - 1 = 0.
\end{equation}
Solving the quadratic equation and eliminating from the previous equations, we arrive at \cite[Equation (4)]{BW85}
\begin{equation}\label{eq:loopless-u}
	A(x,y) = \frac{(1 + q)^2}{2qy^2}\left(y + 3qy - (1 + q)^2 - (1 + q)(1 + q - y)\sqrt{1 - \frac{4qy}{(1 + q)^2}} \right),
\end{equation}
where the sign in front of the square root is taken so that $A(x,y)$ has non-negative coefficients.
From the bijection between loopless maps and 3-connected cubic maps it follows that
\begin{equation}\label{eq:MA}
	M(z,u) = zu^2(A(z,u)-1) = u^3z^2 + (u^3+2u^4)z^3 + (3u^3 +5u^4 + 5u^5)z^4 + \cdots.
\end{equation}

Our next result extends \eqref{sys:cubic} by considering the degree of the root face as an additional parameter.
Let $D(z,u)$ be the generating function of the class $\mathcal{D}$ of cubic maps defined in Section \ref{sec:prelim}, where $u$ marks de degree of the root face, and similarly for $L_1$, $L_2$, $I$, $S$, $P$ and $H$.

\begin{lemma}\label{lem:systemDegree}
	Let $D(z)=D(z,1)$ and $I(z) = I(z,1)$ the univariate series as in \eqref{sys:cubic}.
	Then the following equations hold:
	\begin{equation}\label{eq:systemDegree}
	\renewcommand\arraystretch{1.56}
	\begin{array}{ll}
		C(z,u) 		&= D(z,u) + I(z,u),\\
		D(z,u) 		&= L_1(z,u) + L_2(z,u) + S(z,u) + P(z,u) + H(z,u), \\
		L_1(z,u) 	&= zu(1 + D(z) + I(z)), \\
		L_2(z,u) 	&= zu^4(1 + D(z,u) + uI(z,u)), \\
		I(z,u) 		&= \ds\frac{L_2(z,u)^2}{zu^4},\\
		S(z,u) 		&= D(z,u)(D(z,u) - S(z,u)), \\
		P(z,u) 		&= zu^2(1 + D(z))(1 + D(z,u)), \\
		H(z,u)  	&= \ds\frac{M\left(z(1 + D(z))^3, \ds\frac{1 + D(z,u)}{1 + D(z)}\right)}{1 + D(z,u)}.
	\end{array}
	\end{equation}
\end{lemma}

\begin{proof}
We revisit \eqref{sys:cubic} and enrich it in order to mark the degree of the root face.
The main differences are the equations for the generating functions  counting loop, parallel and $h$-maps.

The series $L_1$ and $L_2$ count loop maps in which the root face has size one and at least two, respectively.
For instance, the maps counted by $L_1$ are obtained from the dumbbell rooted at a loop, encoded by $zu$, in which the non-root loop is possibly replaced by an arbitrary map, hence the factor $1 + D(z) + I(z)$.
Thoses series are univariate since the maps attached to the non-root loop do not contribute to the degree of the root face.
For $L_2$ however, the dumbbell (rooted at a loop) is now encoded by $zu^4$ and the root face degree of the attached map contribute to the total degree.

In the equation for $I$, the difference with the univariate case is that one can only attach loop maps whose root face degree contributes to the total degree, i.e. those counted by $L_2$.
In the equation for $P$, the 3-bond is now encoded by $z^2u$.
And maps attached to the edge directly to the right of the root edge contribute to the root face degree, while maps attached to the left do not.
Finally, in the equation for $H$ every non-root internal edge of a 3-connected cubic map is possibly replaced by a non-isthmus map whose root face degree is not marked.
Whereas, the maps attached to the external edges contribute to the resulting root face degree.
\end{proof}

Next we analyse the singularities of both $M(z,u)$ and $D(z,u)$.
We remark that the condition $|u|\le 1$ and $u$ close enough to the real axis is sufficient to determine the probability generating function $p(u)$ (see the proof of Theorem \ref{thm:root_degree} below) by analytic continuation.
It could have been replaced by a different condition for $u$ close to~1, but this one is convenient for the proof.

\begin{lemma}\label{lem:univ_sing_M}
Suppose that $z$ and $u$ are complex numbers sufficiently close to the real axis and that $|u|\le 1$.
Then the singularity of $M(z,u)$ does not depend on $u$ and is equal to $\tau = 27/256$.
Furtermore, for fixed $u\sim 1$ the Puiseux expansion of $M(z,u)$ for $z\sim\tau$ is of the form
\begin{equation*}
	M(z,u) = M_0(u) + M_2(u)Z^2 + M_3(u)Z^3 + O(Z^4), \qquad Z=\sqrt{1 - z/\tau},
\end{equation*}
where $M_0(u)$, $M_2(u)$ and $M_3(u)$ are algebraic functions, analytic for $|u|\le 1$ sufficiently close to the positive real axis.
\end{lemma}

\begin{proof}	
Given the expression in Equation \eqref{eq:loopless-u}, the singularities of $A(x,y)$, hence those of $M(z,u)$, have only two possible sources:
a) those coming from $u$, or
b) the vanishing of the term $1-4uy/(1+u)^2$ inside the square-root.
We can rule out source b) easily as follows.
A simple calculation shows that function $4u(x)/(1 + u(x))^2$ is increasing for $x\ge0$ and its maximum is at the radius of convergence $\tau$, where it takes the value $3/4$.
Since $|y|\le 1$, for $x$ and $y$ sufficiently close to the positive real line, we have $\left|\frac{4uy}{(1+u)^2}\right| < 1$.
\end{proof}

The analogous statement for $D(z,u)$ needs more work.

\begin{lemma}\label{lem:Dzu_sing}
Suppose that $z$ and $u$ are sufficiently close to the positive reals and that $|u|\le 1$.
Then the dominant singularity of $D(z,u)$ does not depend on $u$ and is equal to $\sigma = \sqrt{3}/36$.
Furtermore, for fixed $u\sim 1$ the Puiseux expansion for $z\sim \sigma$ is of the form
\begin{equation*}
	D(z,u) = D_0(u) + D_2(u)Z^2 + D_3(u)Z^3 + O(Z^4), \qquad Z = \sqrt{1 - z/\sigma},
\end{equation*}
where $D_0(u)$, $D_2(u)$ and $D_3(u)$ are algebraic functions, analytic for $|u|\le 1$ sufficiently close to the positive real axis.
\end{lemma}

\begin{proof}
Eliminating from \eqref{eq:systemDegree}, we get
\begin{equation}\label{eq:DM}
	2uD = 2uM + (1 + D)(1 - \sqrt{1 - 4zu^5(1 + D)})
	+ 2zu^2(1 + D)(1 + D_1 + I) + 2zu^3(1 + D)^2(1 + D_1),
\end{equation}
where $M = M(z(1 + D_1)^3,(1 + D)/(1 + D_1))$, $D = D(z,u)$, $D_1 = D(z)$ and $I = I(z)$.
Observe that the singularities of $D(z,u)$ can either arise from the square-root term, i.e. when $4zu^5(1 + D) = 1$, from a branch point, i.e. a zero of the derivative of \eqref{eq:DM} with respect to $D$, or from the singularities of $D_1$, $I$ and $M$.

We will first rule out any singularity coming from the term $\sqrt{1 - 4zu^5(1 + D)}$.
Adapting the proof of Theorem \ref{thm:enum}(a), we observe that $\sigma$ is the dominant singularity of $D(z)$ and that
\begin{equation}\label{eq:bound_D}
	D(\sigma,1) = D(\sigma) = \frac{3}{4}\sqrt{3} - 1.
\end{equation}
Because its coefficients are non-negative, $D(z,u)$ is increasing in both variables on $(0,\sigma]\times(0,1]$.
Hence if $|u|\le 1$ then $|D(z,u)| \le D(|z|,|u|) \le D(|z|,1)$ converges when $|z| < \sigma\approx 0.04811$.
But then assuming $4zu^5(1 + D(z,u)) = 1$, we get a contradiction using \eqref{eq:bound_D}:
\begin{align*}
	|z| = \frac{1}{|4u^5(1 + D(z,u))|} > \frac{1}{4 + 4D(\sigma,1)} = \frac{1}{3\sqrt{3}}\approx 0.19245.
\end{align*}

Next, we rule out the possibility of a branch point coming from \eqref{eq:DM}.
The derivative of \eqref{eq:DM} with respect to $D$ can be written as follows
\begin{equation}\label{eq:branchpoint}
	2u = 2u\frac{M_2}{1 + D_1} + 1 - \sqrt{1 - 4zu^5(1 + D)} + \frac{2zu^5(1 + D)}{\sqrt{1 - 4zu^5(1 + D)}}
	+ 2zu^2(1 + D_1 + I) + 4zu^3(1 + D)(1 + D_1),
\end{equation}
where $M_2 = M_2(z,u) = (\partial/\partial u)M(z(1 + D_1)^3,(1 + D)/(1 + D_1))$ can be computed from \eqref{eq:loopless} and \eqref{eq:MA}, and verifies
\begin{equation}\label{eq:diffDM}
	\frac{\partial}{\partial u}M(z,u) = \frac{zu(A - 2)(2uA - A_1) - (u - 2)(A - 1)}{2Azu^2 - A_1zu - u + 1},
\end{equation}
with $A = A(z,u)$ and $A_1 = A(z)$.
We assume that there exists a pair $(z_0,u_0)$ with $|z_0| \le \sigma $ and $|u_0| \le 1$ and sufficiently close to the real plane which satisfies \eqref{eq:branchpoint}, and then reach a contradiction.

Both $D(z)$ and $I(z)$ have non-negative coefficients and are thus increasing functions on $(0,\sigma]$.
Hence, as a byproduct of Theorem \ref{thm:enum}(a) we get
\begin{align}
	& |I(z_0)| \le I(|z_0|) \le I(\sigma) = 21\sqrt{3}/4, \label{eq:bound_I} \\
	& |z_0(1 + D(z_0))^3| \le |z_0|(1 + D(|z_0|))^3 \le \sigma(1 + D(\sigma))^3 = \tau, \label{eq:critical_scheme}
\end{align}
where the last equality is a so-called \textit{critical} composition scheme.
Notice also that as $|u_0|\le 1$ we have
\begin{equation}\label{eq:bound_sec_var}
	\left|\frac{1 + D(z_0,u_0)}{1 + D(z_0)}\right| \le 1.
\end{equation}
Further remark that the coefficients of $M_2(z,u)/(1 + D_1(z))$ are also non-negative integers, as they count 3-connected cubic maps where an additional edge of the root face is distinguished and in which every edge but the distinguished one is possibly replaced by a non-isthmus cubic map.
It is thus an increasing function over $(0,\sigma]\times(0,1]$.
And using \eqref{eq:critical_scheme} and \eqref{eq:bound_sec_var} we get
\begin{align}\label{eq:bounds_M2}
	\left|\frac{M_2(z_0,u_0)}{1 + D(z_0)}\right|
	\le \frac{1}{1 + D(|z_0|)} \frac{\partial}{\partial u}M\left(|z_0(1 + D(z_0))^3|, \left|\frac{1 + D(z_0,u_0)}{1 + D(z_0)}\right|\right)
	\le \frac{1}{1 + D(\sigma)}\frac{\partial}{\partial u}M\left(\tau, 1\right) = \frac{\sqrt{3}}{32},
\end{align}
where the last equality is computed from \eqref{eq:bound_D} and \eqref{eq:diffDM}, using the value $A(\tau,1) = A(\tau) = 32/27$ obtained from \eqref{eq:loopless_univ}.
From the right hand-side of \eqref{eq:branchpoint} it now remains to consider the generating function
\begin{align*}
	F(z,u) =\frac{1}{2u} \left(1 - \sqrt{1 - 4zu^5(1 + D)} + \frac{2zu^5(1 + D)}{\sqrt{1 - 4zu^5(1 + D)}}
	+ 2zu^2(1 + D_1 + I) + 4zu^3(1 + D)(1 + D_1)\right).
\end{align*}
Since every series within the  brackets has non-negative coefficients $(z,u)$ is  increasing on $(0,\sigma]\times(0,1]$.
Using \eqref{eq:bound_D} and \eqref{eq:bound_I} we obtain
\begin{align}\label{eq:bounds_F}
	\left|F(z_0,u_0)\right|
	\le F(|z_0|,|u_0|)
	\le F(\sigma,1)
	= 1 - \frac{35\sqrt{3}}{96}.
\end{align}
Finally, plugging \eqref{eq:bounds_M2} and \eqref{eq:bounds_F} together in \eqref{eq:branchpoint} we reach a contradiction:
\begin{align*}
	1 = \left|\frac{M_2(z_0,u_0)}{1 + D(z_0)} + \frac{F(z_0,u_0)}{2u_0}\right|
	\le 1 - \frac{\sqrt{3}}{3}
	< 0.423.
\end{align*}

As mentioned above, this means that the dominant singularity of $D(z,u)$ is that of $D(z)$, $I(z)$ and $M(z,u)$.
Its singular behaviour can then be deduced from Lemma \ref{lem:univ_sing_M} together with Lemma \ref{lem:transfer_sing}.
Note that the functions $D_0(u)$, $D_2(u)$ and $D_3(u)$ are the first coefficients of the Puiseux expansion of $D(z,u)$ for $z\sim\sigma$, which can be computed from the minimal polynomial of $D(z,u)$ of degree 9 and obtained by elimination from \eqref{eq:systemDegree}.
This concludes the proof.
\end{proof}

\paragraph{Proof of Theorem \ref{thm:root_degree}.} %

Using the singular expansion of $D(z,u)$ and the equations in \eqref{eq:systemDegree} we obtain an analogous expansion for $I(z,u)$ and thus for $C(z,u)$ when $|u|\le 1$ and $u$ is sufficiently close to the real line:
\begin{equation*}
	C(z,u) = D(z,u)+ I(z,u) = C_0(u) + C_2(u)Z^2 + C_3(u)Z^3 + O(Z^4),  \qquad Z=\sqrt{1-z/\sigma}.
\end{equation*}

Then we have
\begin{equation*}
	p_k = \lim_{n \to \infty} \frac{[z^n][u^k]C(z,u)}{ [z^n]C(z)}.
\end{equation*}
It follows that the probability generating function is equal to
\begin{equation*}
	p(u) = \sum_k p_ku^k =  \frac{C_3(u)}{C_3(1)}.
\end{equation*}
We observe that $p(u)$ is uniquely determined by analytic continuation.
Furthermore, using \texttt{Maple} we obtain that $p(u)$ is
the unique power series with non-negative coefficients which is a solution of the irreducible polynomial
\begin{equation}\label{poly:p}
Q(p,u)=	a_0(u) + a_1(u)p + a_3(u)p^3,
\end{equation}
where $a_0(u)$, $a_1(u)$ and $a_3(u)$ are the polynomials given in the statement of Theorem \ref{thm:root_degree}.

The dominant singularity $u_0$ of $p(u)$ is computed from the discriminant of $Q(p,u)$ with respect to $p$.
It is the unique real root of
\begin{equation*}
	13u^3 + (4\sqrt{3} - 36)u^2 + + (78-26\sqrt{3})u + 24\sqrt{3} - 60  = 0,
\end{equation*}
and we have  $u_0\approx 1.10254$. The former equation can be rationalized and is equivalent to
\begin{equation*}
	13\,{u}^{6} - 72\,{u}^{5} + 252\,{u}^{4} - 504\,{u}^{3} + 600\,{u}^{2} - 432\,u + 144.
\end{equation*}
The Puiseux expansion of $p(u)$ for $u\sim u_0$ is of the form
\begin{equation*}
	C_3(u) = aU^{-3}+ O(U^{-1}), \qquad \text{with } a\approx 0.028650 \text{ and } U = \sqrt{1-u/u_0}.
\end{equation*}	
Using the Transfer Theorem we finally obtain the estimate
\begin{equation*}
 	p_k \sim c \cdot k^{1/2} q^{k},
\end{equation*}
where $c=a/\kappa(3/2)$ and  $q=u_0^{-1}$.
Let us remark that $u_0 >1$ and $q<1$, in accordance to the fact that the $p_k$ are the tail of a probability distribution. \qed

\paragraph{The maximum face degree.}

Let $p_k$ be as before, and let $p^*_k$ be the limiting probability that a random face has degree $k$.
A double counting argument \cite{L99} shows that  the two distributions are related by
\begin{equation*}
	kp^*_k =\mu p_k,
\end{equation*}
where $\mu$ is the expected degree of a random face (notice that $\sum k p^*_k = \mu)$.
It follows that
\begin{equation*}
	p^*_k \sim  c^* k^{-1/2} q^{k}, \quad c^*=c\mu.
\end{equation*}
	
Let $Y_{n,k}$ be the number  the number of faces of degree at least $k+1$ in maps of size $n$.
As discussed in~\cite{DGN11}, in this situation one has
\begin{equation*}
	\ex Y_{n,k} \approx  \frac{c^*q}{1-q}k^{-1/2}q^k n.
\end{equation*}
Denote by $\Delta_n$ the maximum degree of a random cubic map.
Then we have
\begin{equation*}
	\PP(\Delta_n > k) =\PP( Y_{n,k}>0)\le \ex Y_{n,k}.
\end{equation*}
Thus, if $k^{-1/2}q^kn \to 0$, then $\Delta_n \le k$ almost surely when $n\to\infty$.
This happens if $k =(1+\epsilon) \log n /\log (1/q)$.
Usually such a threshold is tight, so one can expect the converse statement  also to be true.
This would imply that $\Delta_n \sim \log n/  \log (1/q)$.
In order to prove this rigorously, one needs to estimate the variance of $Y_{n,k}$ then apply the second moment method.
This can be achieved by analysing the degree of a second root face (see \cite{DGN11} for details in a similar situation). This program could in principle be carried out by extending Lemma \ref{lem:systemDegree} to mark a second face with a new variable $v$. After verifying the conditions of \cite[Theorem 1.1]{DGN11} one would obtain
\begin{equation*}
	\frac{\Delta_n}{\log n} \to \log(1/q), \qquad \ex \Delta_n \sim 	\frac{1}{\log(1/q)} \log n.
\end{equation*}
Although we expect that the former estimates hold,
we have refrained from doing the necessary  lengthy calculations.


\section{Largest components}\label{sec:components}

This section is devoted to proving Theorems \ref{thm:largest_graph_block}, \ref{thm:largest_cubic_block} and \ref{thm:largest_3comp}.
The proof strategy follows the approach from~\cite{BFSS01}, which works for many different classes of maps (see \cite[Table 4]{BFSS01}).
It consists in first proving a map-Airy limiting distribution for the component (block, cubic block and 3-connected) containing the root edge, the cores defined in Section \ref{sec:prelim}, then transferring it to the largest component via a double-counting argument.
This strategy (see \cite[Appendix D]{BFSS01} or \cite{GW99}) works by rooting maps at a secondary  edge and then  `exchanging the role of the two roots', so that  one can relate the number of maps whose core has size $t$ with those whose largest component has size $t$.

However, the last step does not extend directly to the size of the largest cubic block or the largest 3-connected component of a random cubic map, that is, in the proofs of Theorems \ref{thm:largest_cubic_block} and \ref{thm:largest_3comp}.
The reason is twofold.
First, when counting maps by faces or vertices the argument  from \cite[Appendix D]{BFSS01} fails, as rooting a cubic map at a face or at a vertex does not carry sufficient information, unlike rooting at an edge.
Secondly, the recursive decomposition of cubic maps based on replacements of  edges in the core  has the particularity that each replacement increases the number of edges in the core by one. Thus, when counting cubic maps by edges, one needs to account for this fact in order to encode the number of edges of the largest component containing the root.

Our solution is to introduce an extra variable $u$ that encodes the number $m$ of (non-empty) edge replacements in the core.
In our context, each replacement is accounted for by subdividing the edge once, thus creating a vertex of degree two.
Adapting the proof method developed in \cite{BFSS01}, we show a limit law of the map-Airy type, with fluctuations of order $O(n^{2/3})$, for the number of edges of the core taking into account vertices of degree two.
Finally, we transfer this result to the size of the core without vertices of degree two by showing that the fluctuations of $m$ are typically Gaussian, of order $O(n^{1/2 + \varepsilon})$ for some $\varepsilon > 0$.
The results in this section are thus obtained for cubic maps counted by edges, but the same limiting distributions and constants hold when considering cubic maps counted by faces.

\paragraph{Pure periodicities.}

We note that when counting cubic planar maps by edges one has to take care of so-called {\it pure periodicities}, that is, the parameters satisfy several congruence relations.
More precisely, the number of edges $e$ and the number of faces $f$ satisfy the relations $e = 3(f - 2)$ so that $e$ is always a multiple of $3$.
This is reflected by the appearance of singularities that are not located on the positive real line.

For instance, if a generating function $F(z)$ with non-negative coefficients has the property that the only positive coefficients are those whose indices are multiples of $3$,
then we can write $F(z) = \widehat F(z^3)$, for some function $\widehat F(z)$.
Furthermore, in this case the dominant singularity $z = \rho'> 0$ of $\widehat F(z)$ has three singularities of $F(z)$ as natural counterparts, $z_1 = \rho$, $z_2 = \rho e^{i2\pi/3}$ and $z_3 = \rho e^{i4\pi/3}$, where $\rho^3 = \rho'$.
Hence, if we perform Cauchy integration to obtain an asymptotic expansion for the coefficient $[z^n] F(z)$  we have to take into account the contribution of the integral close to $\rho$, $\rho e^{i2\pi/3}$ and $\rho e^{i4\pi/3}$.
Due to the relation between $F(z)$ and $\widehat F(z)$, these three contributions are the same up to a third root of unity.
So that if $n$ is a multiple of $3$ then the total contribution is three times the contribution coming from the singularity $z = \rho$.
On the other hand, if $n$ is not a multiple of $3$ then the three contributions sum up to zero.

In order to make the following analysis more transparent and readable, we  assume that no pure periodicity appears, that is, no congruence relation between the non-zero coefficients is considered.
This means in particular that we do not take into account whether $n$ is a multiple of $3$ or not, and thus neglect the factor $3$.
However, we will eventually compute ratios of the form $([z^n]\, F(z)) / ([z^n]\, G(z))$ so that the factor $3$ finally cancels.
Thanks to this   simplification we just have to consider the positive dominant singularity.
Clearly, all computations can be made completely rigorous.

\subsection{Map-Airy law for the size of the 2-core}

This section is devoted to the proof of a map-Airy law for the size (number of edges) of the 2-core of a random cubic planar map with $n$ edges, as $n\to\infty$, but parameterized by a variable $u$ marking the number of vertices of degree two in the 2-core.
Before stating and proving it, we will however need to establish preliminary results.

Our first preliminary result allows us to encode, using trivariate generating functions, the number of edges of the 2-core of a cubic map while keeping track of the number of vertices of degree two.

\begin{lemma}\label{lem:decomp_2core}
	Let $B(y)$ be the generating function of 2-connected cubic maps and $C(z,w,u)$ that of cubic maps, where  $y$ and $z$ mark, respectively, the total number of edges, while $w$ and $u$ mark, respectively, the number of edges and  vertices of degree two in the 2-cores.
	Furthermore, let $L(z)$ be the generating function of loop cubic maps where  $z$  marks only  non-root edges.
	Then the following equation holds:
	\begin{equation}\label{eq:c^*core}
		C(z,w,u) = B\left(\frac{zw}{1 - zwuL(z)}\right)\frac{1}{1 - zwuL(z)}
			+ \frac{zwuL(z)}{1 - zwuL(z)} + \frac{L(z)^2}{4z}.
	\end{equation}
\end{lemma}

\begin{proof}
After iteratively removing all the cherries from a cubic map, we are left with three possible configurations: the resulting map is either a dumbbell (1), a rooted cycle of length at least one (2), or is 2-connected and is thus the 2-core (3).

Case (1): cubic maps in this case are counted by $L(z)^2/4z$.
Dumbbells are not 2-connected and thus there is no occurrence of the variables $w$ and $u$ here.

Case (2): those cubic maps can be derived by attaching a loop map at each vertex of a rooted cycle of size at least one.
Attaching a loop map at a vertex $v$ is done by removing its root edge and identifying its root vertex with $v$.
Each such vertex (originally of degree two) amounts for one in the size of the 2-core.
So that the generating function for this family of cubic maps is $zwuL(z)/(1 - zwuL(z))$.

Case (3): any 2-core of some cubic map can be obtained by replacing the edges of some 2-connected cubic map by (possibly empty) paths.
The length of each added path contributes to as many vertices of degree two in the 2-core.
Conversely, one recovers the original map from its 2-core by attaching a loop map at each vertex of degree two.
Such maps are thus encoded by $B\left(zw/(1 - zwuL(z))\right)/(1 - zwuL(z))$.
The factor $1/(1 - zwuL(z))$ amounts for the extra re-rooting of the map when the root edge of the core was replaced by a non-empty path.
\end{proof}

In order to adapt the methods of \cite{BKW21} to the composition scheme in \eqref{eq:c^*core}, we will need a notion of 'criticality'.
This is our next preliminary result.

\paragraph{Critical composition scheme.}

When considering the first summand of the right side of \eqref{eq:c^*core}, i.e. of the composition scheme, the variables $z$ and $w$ always appear together as $zw$ since an edge in the 2-core contributes to the total number of edges.
This motivates the changes of variables
\begin{align*}
	x = zw
	\qquad\text{and}\qquad
	v = uL(z),
\end{align*}
which transforms \eqref{eq:c^*core} into
\begin{equation*}
	C(x,v) = B\left(\frac{x}{1 - xv}\right)\cdot\frac{1}{1 - xv} + \frac{xv}{1 - xv} + \frac{L(z)^2}{4z}.
\end{equation*}
Thus, for $t > 0$ we have
\begin{equation}\label{eq:numer_distributionXstar}
	[w^t]\, C(z,w,u)
	= [x^t]\, B\left(\frac{x}{1 - xv}\right)\cdot\frac{1}{1 - xv}z^t + v^tz^t,
\end{equation}

We will eventually see that only the first summand of \eqref{eq:numer_distributionXstar} contributes to the total mass of the distributions of $X_n$ and $X_n^*$, the random variables in Theorems \ref{thm:largest_graph_block} and \ref{thm:largest_cubic_block} respectively.
Hence, we can safely restrict our study to the composition scheme
\begin{equation*}
	\widetilde C(z,w,u) = \widetilde  C(x,v) = B\left(\frac{x}{1 - xv}\right)\cdot\frac{1}{1 - xv},
\end{equation*}
and define $x(v) = \tau/(1+ \tau v)$ so that
\begin{equation}\label{eq:x(v)_tau}
	\frac{x(v)}{1 - x(v)v} = \tau
	\quad\text{and}\quad
	x(v)^{-1} = \tau^{-1} + v.
\end{equation}
This composition scheme is `critical' in the sense that, when $z\sim\rho$ and $w=u=1$, we have
\begin{equation*}
	v = uL(z) \sim L_0, \quad x(v)\sim\rho,
	\quad\text{and}\quad
	\frac{zw}{1-zwu L(z)} = \frac{x}{1-xv} \sim \frac{\rho}{1 - \rho L_0} = \tau,
\end{equation*}
that is,
\begin{equation*}
	\rho = \frac{\tau}{1 + \tau L_0}
	\quad\text{and}\quad
	\rho^{-1} = \tau^{-1} + L_0.
\end{equation*}

\medskip

Our last preliminary result is to derive asymtptotic estimates for the number of cubic maps, analogue to those in Section \ref{sec:enum}, but this time counted by edges.

\paragraph{Estimates for cubic maps counted by edges.}

We already mentioned that there are three times as many edges as faces minus two in a cubic map.
Hence, the corresponding generating functions and their dominant singular behaviour can be obtained by applying the change of variable $z\to z^3$ in the proofs of Theorem \ref{thm:enum} (a) and (b).

The generating function $B(y)$ (where the exponent of $y$ takes care of the number of edges) has its dominant singularity at $\tau = 2^{1/3}/3$ and for $y\sim\tau$ we have the expansion
\begin{equation}\label{eq:expansion_B}
    B(y) = B_0 - B_2Y^2 + B_3Y^3 + O(Y^4), \qquad Y = \sqrt{1 - \frac{y}{\tau}},
\end{equation}
with $B_0 = 1/8$, $B_2 = 9/8$ and $B_3 = 3$.
There are corresponding singularities at  $\tau e^{i2\pi/3}$ and $\tau e^{i4\pi/3}$ that we neglect.
So in what follows we assume that $B(y)$ is $\Delta$-analytic so that the Transfer Theorem implies
\begin{equation}\label{eq:estimate_B_edges}
    [y^n] B(y) = \frac{3 B_3}{4\sqrt \pi} \, n^{-5/2}\tau^{-n}(1 + o(1)),
    \qquad \text{as } n\to\infty.
\end{equation}
As mentioned above we omit a factor $3$ and the restriction to $n$ that are multiples of $3$.

Similarly we assume that, both $C(z)$ and $L(z)$ are $\Delta$-analytic at $\rho = 2^{1/3}\sqrt{3}/6$.
And for $z\sim\rho$ and $Z = \sqrt{1 - z/\rho}$ it holds that
\begin{equation}\label{eq:expansions_LC}
    C(z) = C_0 - C_2Z^2 + C_3Z^3 + O(Z^4)
    \quad\text{and}\quad
    L(z) = L_0 - L_2Z^2 + L_3Z^3 + O(Z^4),
\end{equation}
with $C_0 = 6\sqrt{3} - 10$, $C_2 = 18(2 - \sqrt{3})$, $C_3 = 12\sqrt{2}$, and $L_0 = 2^{2/3}(\sqrt{3} - 3/2)$, $L_2 = 2^{2/3}(3 - \sqrt{3})$, $L_3 = 2^{1/6}4$.
In particular, this means that the number of rooted cubic maps with $n$ edges is
\begin{equation}\label{eq:estimate_C_edges}
    [z^n]\, C(z) = \frac{3 C_3}{4\sqrt \pi} \, n^{-5/2}\rho^{-n}(1+o(1)),
    \qquad \text{as } n\to\infty,
    \quad\text{and with } \rho = 2^{1/3}\sqrt{3}/6.
\end{equation}
Again we omit here a factor $3$ and the restriction to $n$ that are multiples of $3$.

\medskip

We are now in a position to state and prove the main result of this section.

\begin{proposition}\label{prop:airy_2core}
Let $\alpha_0 >0$ (that will be fixed later), $q = O(1)$ and $0 < \varepsilon < 1/6$.
Then for $t = \alpha_0n + qn^{2/3}$ and $u = 1 + O(t^{-1/2 + \varepsilon})$ the following holds as $n\to\infty$
\begin{equation}\label{eq:airy_estimate}
	[z^nw^t]\, \widetilde C(z,w,u) = \frac{3B_3}{4\sqrt{\pi}}(1 + \tau uL_0)^{5/2}
	\rho^{-n+t}\left(\tau ^{-1} + uL_0\right)^t n^{-5/2}\alpha_0^{-3/2} n^{-2/3} c\A(cq) (1 + o(1)),
\end{equation}
where $c = \left(3L_3/L_2\right)^{2/3}\alpha_0^{-1}(1 - \alpha_0)^{-2/3}$.
\end{proposition}

\begin{proof}
From the notation in \eqref{eq:x(v)_tau}, let us set
\begin{equation*}
	X = \sqrt{1 - \frac{x}{x(v)}},
	\quad\text{so that}\quad
	x = x(v)(1 - X^2).
\end{equation*}
Thus, when $y = x/(1 - xv)\sim\tau$ and considering $Y$ from \eqref{eq:expansion_B}, we can write
\begin{equation}\label{eq:YtoX}
	Y^2 =  1 - \frac y{\tau} = 1 - \frac{\frac{x}{1 - xv}}{\tau}
	= 1 - \frac{\frac{x(v)(1 - X^2)}{1 - x(v)(1 - X^2)v}}{\frac{x(v)}{1 - x(v)v}}
	= \frac{X^2}{1 - x(v)v} + O(X^4).
\end{equation}
So that the following asymptotic estimates hold
\begin{equation*}\label{eq:amenable_cauchy_formula}
	\begin{split}
	[x^t]\, \widetilde C(x,v)z^t
	& \sim [x^t]\, B_3 Y^3 \frac{1}{1 - xv} z^t,
		\qquad\text{ by applying the Transfer Theorem to } \eqref{eq:expansion_B} \text{ as } y\sim\tau,  \\
	& \sim [x^t]\, B_3 \frac{X^3}{(1 - x(v)v)^{3/2}} \frac{1}{1 - xv} z^t,
		\qquad\text{using } \eqref{eq:YtoX}, \\
	& \sim B_3 \frac{3}{4\sqrt{\pi}} (1 - x(v)v)^{-5/2} x(v)^{-t} t^{-5/2} z^t,
		\qquad\text{by the Transfer Theorem as } t\to\infty, \\
	& = B_3 \frac{3}{4\sqrt{\pi}}(1 + \tau uL(z))^{5/2} \left(\tau^{-1} + uL(z)\right)^t t^{-5/2} z^t,
		\qquad\text{using } \eqref{eq:x(v)_tau}.
	\end{split}
\end{equation*}
Hence,
\begin{equation*}
	[w^t]\, C(z,w,u) \sim  [x^t]\, \widetilde C(x,uL(z))z^t
	\sim B_3 \frac{3}{4\sqrt{\pi}} (1 + \tau v)^{5/2} \left(\tau^{-1} + v\right)^t t^{-5/2} z^t.
\end{equation*}

Looking at each term of the above estimate, we first obtain using \eqref{eq:expansions_LC} and as $z\sim\rho$
\begin{equation}\label{eq:asymptotic_main_term}
	\begin{split}
	\left(\tau^{-1} + uL(z)\right)^t
	& \sim \left(\tau ^{-1} + uL_0 - uL_2Z^2 + uL_3Z^3\right)^t \\
	& \sim \left(\tau ^{-1} + uL_0\right)^t \exp\left(-\frac{uL_2}{\tau ^{-1} + uL_0}tZ^2 + \frac{uL_3}{\tau ^{-1} + uL_0}tZ^3\right).
	\end{split}
\end{equation}
And similarly, as $z\sim\rho$ we get
\begin{align}\label{eq:asymptotic_extra_term_5/2}
	\begin{split}
	\left(1 + \tau uL(z)\right)^{5/2}
	& \sim (1 + \tau uL_0)^{5/2} \exp\left(-\frac{\tau uL_2}{1 + \tau uL_0}\frac{5}{2}Z^2 + \frac{\tau uL_3}
	 {1 + \tau uL_0}\frac{5}{2}Z^3\right).
	\end{split}
\end{align}

With this at hand, we are in a position to compute the coefficient $[z^n][x^t]\, \widetilde C(x,uL(z))z^t$ via Cauchy's formula, using a suitable contour $\gamma$ in the $z$-plane that will be defined later:
\begin{equation}\label{eq:cauchy}
	[z^n][x^t]\, \widetilde C(x,uL(z))z^t = [z^{n-t}][x^t]\, \widetilde C(x,uL(z) )
	= \frac 1{2i\pi} \int_{\gamma} \frac{[x^t]\widetilde C(x, uL(z))}{z^{n-t}} \frac{dz}{z}.
\end{equation}
In fact, as argued in the proof of \cite[Theorem 5(ii)]{BFSS01}, the asymptotically significant part of \eqref{eq:cauchy} arises when $z$ is at distance $(n-t)^{-2/3}$ from $\rho$, for instance when $Z^2$ in \eqref{eq:asymptotic_main_term} scales with $(n-t)^{-2/3}$.
This justifies the following definition of the contour $\gamma$, depicted in Figure \ref{fig:contour}.
It includes a positively oriented "loop" that is made of two rays at an angle of $\pi/3$ and $-\pi/3$ with $(0,+\infty)$, and intersecting on the real axis at distance $\rho n^{-2/3}$ of $\rho$ from the left.
This contour is called $\gamma_2$ in the proof of \cite[Theorem 5(ii)]{BFSS01}.
\begin{figure}
	\begin{center}
		\includegraphics[scale=1]{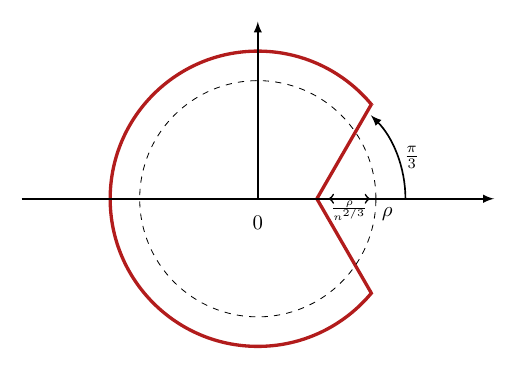}
	\end{center}
	\caption{The contour of integration $\gamma$ in the $z$-plane.}
	\label{fig:contour}
\end{figure}
Furthermore, this motivates the following change of variable
\begin{equation}\label{eq:change_variable_cauchy}
	s = s(z) = (n - t)^{2/3}(1 - z/\rho).
\end{equation}
So that
\begin{align*}
	Z^2 = \frac{s}{(n-t)^{2/3}}, \quad
	z = z(s) = \rho\left(1 - \frac{s}{(n-t)^{2/3}}\right)
	\quad\text{and}\quad
	dz = -\frac{\rho}{(n-t)^{2/3}}ds.
\end{align*}
Consequently, this changes \eqref{eq:asymptotic_main_term} into
\begin{equation}\label{eq:tau_L_cauchy}
	\left(\tau^{-1} + uL(z)\right)^t \sim \left(\tau ^{-1} + uL_0\right)^t
	\exp\left(-\frac{uL_2}{\tau ^{-1} + uL_0}\frac{ts}{(n-t)^{2/3}}
	+ \frac{uL_3}{\tau ^{-1} + uL_0}\frac{ts^{3/2}}{n-t}\right).
\end{equation}
While \eqref{eq:asymptotic_extra_term_5/2} becomes
\begin{align*}
	\left(1 + \tau uL(z)\right)^{5/2} \sim (1 + \tau uL_0)^{5/2}.
\end{align*}
And the \textit{shifted Cauchy kernel} is also transformed into
\begin{equation}\label{eq:cauchy_kernel}
	z^{-(n-t)} = \rho^{-(n-t)} e^{s(n-t)^{1/3} + o(n^{-1/3})}.
\end{equation}

Now, from $t = \alpha_0n + qn^{2/3}$ we get the following estimates as $n\to\infty$
\begin{align*}
	& n - t = (1 - \alpha_0)n - qn^{2/3},
	\qquad\qquad\qquad t^{-5/2} = \alpha_0^{-5/2}n^{-5/2} + O(n^{-17/6}), \\
	& (n-t)^{-1} = \frac{1}{1 - \alpha_0}n^{-1} + O(n^{-4/3}),
	\qquad (n-t)^{-2/3} = \frac{1}{(1 - \alpha_0)^{2/3}}n^{-2/3} + O(n^{-1/3}), \\
	& t(n-t)^{-1} = \frac{\alpha_0}{1 - \alpha_0} + O(n^{-1/3}),
	\qquad\quad (n-t)^{1/3} = (1 - \alpha_0)^{1/3}n^{1/3} - \frac{q}{3(1 - \alpha_0)^{2/3}} + O(n^{-1}), \\
	& t(n-t)^{-2/3} = \frac{\alpha_0}{(1 - \alpha_0)^{2/3}}n^{1/3}
	+ \left(\frac{2\alpha_0}{3(1 - \alpha_0)^{5/3}} + \frac{1}{(1 - \alpha_0)^{2/3}}\right)q + O(n^{-1/3}).
\end{align*}
Hence, the exponential part of the product of \eqref{eq:tau_L_cauchy} with \eqref{eq:cauchy_kernel} is asymptotically
given by
\begin{equation}\label{eq:exponent}
	\begin{split}
		& s n^{1/3} \left( (1 - \alpha_0)^{1/3} - \frac{uL_2}{\tau ^{-1} + uL_0}\frac{\alpha_0}{(1 - \alpha_0)^{2/3}} \right) \\
		& - qs \left( \frac{1}{3(1 - \alpha_0)^{2/3}} + \frac{uL_2}{\tau^{-1} + uL_0 } \left( \frac{2\alpha_0}{3(1 - \alpha_0)^{5/3} } + \frac 1{(1 - \alpha_0)^{2/3}} \right) \right)
		+ \frac{\alpha_0}{1-\alpha_0}\frac{L_3}{\tau^{-1} + u L_0}s^{3/2}.
	\end{split}
\end{equation}
Notice then that, using the values in \eqref{eq:expansion_B} and \eqref{eq:expansions_LC}, if we set
\begin{equation}\label{eq:value_alpha}
    \alpha_0 = \frac{\tau^{-1} + L_0}{\tau^{-1} + L_0 + L_2} = \frac 1{\sqrt 3},
\end{equation}
this ensures that
\begin{align*}
    (1 - \alpha_0)^{1/3} -  \frac{L_2}{\tau ^{-1} + L_0}\frac{\alpha_0}{(1 - \alpha_0)^{2/3}} = 0
    \quad\text{and}\quad
	\frac{L_2}{\tau^{-1} + L_0} = \frac{1 - \alpha_0}{\alpha_0}.
\end{align*}
And from $u = 1 + O(t^{-1/2 + \varepsilon})$ we further get that \eqref{eq:exponent} reduces asymptotically to
\begin{align*}
	\frac{L_3}{L_2}s^{3/2} - \frac{1}{\alpha_0(1 - \alpha_0)^{2/3}}qs + o(1).
\end{align*}
Remark finally that under the change of variable \eqref{eq:change_variable_cauchy} $\gamma$ evolves into a contour made of two segments of angle $2\pi/3$ and $-2\pi/3$, intersecting at $-1$, and each of length $O(\log^2 n)$.
As argued in the proof of \cite[Theorem 5(ii)]{BFSS01}, this contour can be extended back to infinity at the expense of exponentially small error terms.

\medskip

Taking all the above assumptions into account then reverting the orientation of the new contour and shifting it by one, we can tranform \eqref{eq:cauchy} into
\begin{align*}
	[z^n][w^t]\, \widetilde C(z,w,u)
	\sim \frac{3B_3(1 + \tau uL_0)^{5/2}}{4\sqrt{\pi}}
	\frac{\alpha_0^{-5/2}\rho^{-n+t}\left(\tau ^{-1} + uL_0\right)^t}{(1 - \alpha_0)^{2/3}n^{2/3}n^{5/2}} \frac{1}{2\pi i}
	\int_{\infty e^{-i2\pi/3}}^{\infty e^{i2\pi/3}}
	e^{ \frac{L_3}{L_2}s^{3/2} - \frac{1}{\alpha_0(1 - \alpha_0)^{2/3}}qs } ds.
\end{align*}	
The claimed estimate is then deduced by setting $d_1 = \alpha_0^{-1}(1 - \alpha_0)^{-2/3}$ and $d_2 = L_2^{-1}L_3$ in the integral representation from \eqref{eq:airy_integral_representation}.
\end{proof}


\subsection{Proofs of the main results}

\paragraph{Proof of Theorem \ref{thm:largest_graph_block}.}

In order to study the distribution of $X_n$ in the central regime, that is, when blocks can have vertices of degree two and size in the range $\alpha_0n + O(n^{2/3})$, we first consider the random variable $Y_n$ denoting the size of the 2-core in $\textsf{M}_n$.
Let $q$ be in some bounded interval, then we have
\begin{align*}
	\mathbb{P}\left( Y_n = \lfloor \alpha_0\, n + qn^{2/3} \rfloor \right)
 	& \sim \frac{[z^nw^t]\, \widetilde C(z,w,1)}{[z^n]\, C(z)}
	\sim \frac{B_3}{C_3}(1 + \tau L_0)^{5/2} \alpha_0^{-3/2} n^{-2/3} c\A(cq).
\end{align*}
It follows then from the double-counting lemma used to prove Theorem 7 in \cite[Appendix D]{BFSS01} that, for $t = \alpha_0 n + qn^{2/3}$, there exists some $A<1$ such that
\begin{align*}
	\mathbb{P}\left( X_n = t \right) = \frac{n}{t} \mathbb{P}\left( Y_n = t \right) (1 + O(A^n)).
\end{align*}
Since from \eqref{eq:value_alpha} $\alpha_0 = 1/\sqrt 3$, and using the values in \eqref{eq:expansion_B} and \eqref{eq:expansions_LC}
\begin{equation}\label{eq:cst_one}
	\frac{B_3}{C_3}(1 + \tau L_0)^{5/2}  \alpha_0^{-5/2} = 1,
\end{equation}
it immediately follows that
\begin{equation*}
	\mathbb{P}\left( X_n = \lfloor n/\sqrt 3 + qn^{2/3} \rfloor \right) \sim n^{-2/3} c\A(cy),
\end{equation*}
with $c = 2\sqrt{3}/(1 - 1/\sqrt{3})^{4/3}$, as claimed.
\qed


\medskip

We now focus on proving the first intermediate result required for the proof of Theorem \ref{thm:largest_cubic_block}, that is, in a random cubic map whose 2-core has $t$ edges and $m$ vertices of degree two, $m$ has Gaussian fluctuations centered around $t$.

\begin{lemma}
Fix some $0 < \varepsilon, \varepsilon' < 1/6$, and let as before $t = \alpha_0 n + qn^{2/3}$ and $u = 1 + O(t^{-1/2 + \varepsilon})$, where $\alpha_0 = 1/\sqrt{3}$ and $q = O(1)$.
Further let $m = \beta_0t + r$, for some $\beta_0 > 0$ (that will be fixed later) and $r = O(t^{1/2 + \varepsilon'})$.
Then as $n\to\infty$ we have
 \begin{equation}\label{eq:airy_estimate_with_exp}
	[z^nw^tu^m]\, \widetilde C(z,w,u)
	\sim \frac{3B_3}{4\sqrt{\pi}}\frac{(1 + \tau L_0)^{5/2}}{\alpha_0^{3/2}} \frac{\rho^{-n}}{n^{5/2}}
	\frac{e^{-r^2/(2\sigma^2\alpha_0n)}}{\sqrt{2\pi\sigma^2\alpha_0n}} n^{-2/3} c\A(cq),
\end{equation}
where $c$ is as in \eqref{eq:airy_estimate}.
\end{lemma}

\begin{proof}
By using the identity $\rho = (\tau^{-1} + L_0)^{-1}$, we obtain from \eqref{eq:airy_estimate}
\begin{equation}\label{eq:airy_estimate_with_u}
	[z^nw^t]\, \widetilde C(z,w,u)\sim \frac{3B_3}{4\sqrt{\pi}}\frac{(1 + \tau uL_0)^{5/2}}{\alpha_0^{3/2}} \frac{\rho^{-n}}{ n^{5/2}} \left(\frac{\tau^{-1} + uL_0}{\tau^{-1} + L_0}\right)^t n^{-2/3} c\A(cq).
\end{equation}
Remark that the function
\begin{equation*}
	u\mapsto \left(\frac{\tau^{-1} + uL_0}{\tau^{-1} + L_0}\right)^t
\end{equation*}
can be considered as the probability generating function of the sum of $t$ iid random variables.
In the present context, this leads to the binomial distribution.

A known method to obtain a local limit theorem is to use Cauchy's formula and a saddle point like integration for the contour $|u| = 1$ (see for instance \cite{D94}).
For $r = O(t^{1/2 + \varepsilon'})$, this yields
\begin{equation}\label{eq:concentration_binomial}
	[u^m]\, \left(\frac{\tau^{-1} + uL_0}{\tau^{-1} + L_0}\right)^t
	= \frac 1{2\pi i} \int_{|u| = 1} \left(\frac{\tau^{-1} + uL_0}{\tau^{-1} + L_0}\right)^t  \frac{du}{u^{m+1}}
	\sim \frac{e^{-r^2/(2\sigma^2t)}}{\sqrt{2\pi\sigma^2t}},
\end{equation}
when
\begin{equation*}
	\beta_0 = \frac{L_0}{\tau ^{-1} + L_0} = 1 - \frac{\sqrt 3}{2}
	\text{ and } \sigma^2 = \frac{L_0}{\tau(\tau ^{-1} + L_0)^2}.
\end{equation*}
It is important to note that the asymptotic leading term ${e^{-r^2/(2\sigma^2t)}}/{\sqrt{2\pi\sigma^2t}}$ in \eqref{eq:concentration_binomial} comes from the part of integration where $|u - 1| = O(t^{-1/2 + \varepsilon})$.
This justifies the above assumptions and also shows that the asymptotic leading term of \eqref{eq:airy_estimate_with_u} is indeed equal to the claimed estimate.
\end{proof}

\begin{remark}\label{rem:airy_tails}
	Concerning the other regimes of $t$, outside the central region $t = \alpha_0n + O(n^{2/3})$, of the distribution of $[z^nw^tu^m]\, \widetilde C(z,w,u)/[z^n]\, C(z)$, one can for instance show that the \textit{left tail} decays like $O(t^{-1/2}t^{-3/2})=O(t^{-2})$, when either $t = O(1)$ or $t = \alpha n$ with $\alpha < \alpha_0$, and $m = \beta_0t+ r$.
	Just like for the central regime, this can be derived by adapting the proof of \cite[Theorem 5]{BFSS01}(i) to the present context.
	
	It can similarly be shown (see \cite[Theorem 5]{BFSS01}(iii)) that the \textit{right tail} decays like $O(A^n)$ for some $A<1$, when $t = \alpha n$ with $\alpha > \alpha_0$.
\end{remark}

In order to transfer this asymptotic result on the $2$-core to the largest block, we need a generalisation of the double-counting lemma from \cite[Appendix D]{BFSS01} to account for the vertices of degree two.

\begin{lemma}\label{lem:double-counting_m}
	Let $\alpha_0,\beta_0\in (0,1)$ be as above, let $q = O(1)$, and assume that $r = O(t^{1/2 + \varepsilon})$ for some $0 < \varepsilon < 1/6$.
	Let then $c_{n,t,m}$ be the number of cubic maps with $n$ edges, and a 2-core with $t = \alpha_0 n + qn^{2/3}$ edges and $m = \beta_0t + r$ vertices of degree two.
	Further let $c^*_{n,t,m}$ denote the number of cubic maps with $n$ edges, and a largest block with $t$ edges and $m$ vertices of degree two.
	Then there exists $A < 1$, independent from the choices of $q$ and $r$, such that
	\begin{align*}
		c^*_{n,t,m} = \frac{n}{t}c_{n,t,m}\left(1 + O(A^n)\right).
	\end{align*}
\end{lemma}

\begin{proof}
	The statements in \cite[Appendix D]{BFSS01} translate almost directly to our setting.
	Let $a_{n,t,m}$ be the number of cubic maps whose $2$-core is also the largest block and has $t$ edges and $m$ vertices of degree $2$.
	Similarly, let $b_{n,t,m}$ be the number of cubic maps whose $2$-core has $t$ edges and $m$ vertices of degree $2$, but is not the largest block.
	Then $c_{n,t,m} = a_{n,t,m} + b_{n,t,m}$.
	Furthermore, by rooting each cubic map at a second edge we get the following identity
	\begin{equation}\label{eq:double-counting_m}
		2nc_{n,t,m} = 2na_{n,t,m} + 2nb_{n,t,m} = 2tc^*_{n,t,m} + 2nb_{n,t,m},
	\end{equation}
	where the second equality can be derived as follows.
	Among the $2nc_{n,t,m}$ cubic maps with a 2-core of size $t$ and a second root edge, $2nb_{n,t,m}$ of them have a largest block of size $\ell > t$.
	The remaining ones have a $2$-core which is the largest block and, upon exchanging the r\^ole of the two roots, those maps are identified with the $2tc^*_{n,t,m}$ cubic maps with a largest block $\textsf{L}$ of size $t$ and a secondary root chosen in $\textsf{L}$.

	Let now $\textsf{B}$ be one of the cubic maps counted by $b_{n,t,m}$, that is $\textsf{B}$ has $n$ edges, a $2$-core $\textsf{T}$ of size $t$ with $m$ vertices of degree two, and a largest block $\textsf{L}$ of size $\ell > t$.
	By construction, $\textsf{L}$ is contained in a loop cubic map $\textsf{H}$ of size $h$ attached to $v$, a vertex of degree two of $\textsf{T}$.
	Let $\textsf{D}$ be the cubic map of size $n-h$ obtained from $\textsf{B}$ by removing $\textsf{H}$.
	Then the $2$-core of $\textsf{D}$ is also $\textsf{T}$.

	Conversely, $\textsf{B}$ can be uniquely reconstructed from $\textsf{D}$, $\textsf{H}$ and $v$.
	Thus, $b_{n,t,m}$ is bounded above by the number of such triples: there are $O(c_{n-h,t,m})$ many maps $\textsf{D}$ whose $2$-core of size $t$ contains $m$ vertices $v$ of degree two; and to every $v$ there are $O(c_{h,\ell,m}^*)$ many possible loop cubic maps $\textsf{H}$ to be attached.
	This means that there exists a constant $A_0$ such that
	\begin{align*}
		b_{n,t,m}\le A_0 \sum_{\substack{\ell,h \\ t < \ell < h < n - t}} m c_{n-h,t,m} c_{h,\ell,m}^*
		& \le mA_0 \sum_{\substack{\ell,h \\ t < \ell < h < n - t}} c_{n-h,t,m}\frac{h}{\ell} (c_{h,\ell,m} - b_{h,\ell,m}) \\
		& \le mA_0 \sum_{\substack{\ell,h \\ t < \ell < h < n - t}} \frac{h}{\ell} c_{n-h,t,m} c_{h,\ell,m},
	\end{align*}
	where the second inequality stems from \eqref{eq:double-counting_m}.
	Thus,
	\begin{equation*}
		\frac{b_{n,t,m}}{c_{n,t,m}} \le mA_0 \sum_{\substack{\ell,h \\ t < \ell < h < n - t}}
		\frac{h}{\ell} \frac{c_{n-h,t,m}}{[z^{n-h}]\, C(z)} \frac{c_{h,\ell,m}}{[z^h]\, C(z)} \frac{[z^{n-h}]\, C(z)\cdot [z^h]\, C(z)}{[z^n]\, C(z)} \frac{[z^n]\, C(z)}{c_{n,t,m}}.
	\end{equation*}

	We now separately bound the ratios present in the right hand-side of this inequality as $n\to\infty$.
	First, from \eqref{eq:estimate_C_edges} we get $[z^{n-h}]\, C(z)\cdot[z^h]\, C(z)/[z^n]\, C(z) = O(n^{5/2}h^{-5/2}(n-h)^{-5/2})$.
	Second, when $\ell$ and $h$ are fixed we fall under the regime of the left tail of the distribution of $[z^nw^tu^m]\, \widetilde C(z,w,u)/[z^n]\, C(z)$ (see Remark \ref{rem:airy_tails}), so that $c_{h,\ell,m}/[z^h]\, C(z) = O(\ell^{-1/2}\ell^{-3/2})$.
	While \eqref{eq:airy_estimate_with_exp} implies that $c_{n,t,m}/[z^n]\, C(z) = O(n^{-2/3}n^{-1/2})$.
	Furthermore, we have $t/(n-h) > \alpha_0/(1 - \alpha_0) > \alpha_0$.
	This is the regime of the right tail of the distribution of $[z^nw^tu^m]\, \widetilde C(z,w,u)/[z^n]\, C(z)$ (see again Remark \ref{rem:airy_tails}), thus $c_{n-h,t,m}/[z^{n-h}]\, C(z) = O(A_1^t)$ for some $A_1 < 1$.
	All together, this implies the existence of constants $A_2$ and $A_3$ such that
	\begin{equation}\label{eq:bad_maps}
		\frac{b_{n,t,m}}{c_{n,t,m}}\le mA_2 A_1^t \sum_{\substack{\ell,h \\ t < \ell < h < n - t}}
		\ell^{-3}h^{-3/2}(n-h)^{-5/2}n^{11/3}
		\le A_3A^n.
	\end{equation}
	And the claim follows by combining \eqref{eq:double-counting_m} with \eqref{eq:bad_maps}.
\end{proof}

With this at hand, we can now transfer the estimate \eqref{eq:airy_estimate_with_exp} to the distribution of $X_n^*$.

\paragraph{Proof of Theorem \ref{thm:largest_cubic_block}.}

Let $c_{n,t^*}^*$ denote the number of cubic maps of size $n$ whose largest cubic block has size $t^*$.
By definition, we have
\begin{align*}
	\mathbb{P}\left( X_n^* = t^* \right)
	= \frac{c_{n,t^*}^*}{[z^n]\, C(z)}
	= \frac{1}{[z^n]\, C(z)} \sum_{m,t:\, t-m = t^*} c_{n,t,m}^*.
\end{align*}
And consequently, by Lemma \ref{lem:double-counting_m}
\begin{align*}
	\mathbb{P}\left( X_n^* = t^* \right)
	\sim \frac{1}{[z^n]\, C(z)} \sum_{m,t:\, t-m = t^*} \frac{n}{t} c_{n,t,m}
	\sim \frac{1}{\alpha_0} \frac{1}{[z^n]\, C(z)} \sum_{m,t:\, t-m = t^*} c_{n,t,m}.
\end{align*}

We recall that $n$ and $t$ are (almost) proportional, with an {\it error} of order $n^{2/3}$, and $m$ and $t$ are (almost) proportional too, with an error of order $n^{1/2}$.
Thus, $n$ and $t^* = t-m$ are again (almost) proportional, with an error of order $n^{2/3}$.
And by representing $t^*$ as
\begin{align*}
	t^* = t-m = (1 - \beta_0) t - r
	= \alpha_0(1 - \beta_0) n + q (1 - \beta_0) n^{2/3} - r
	= \alpha_0(1 - \beta_0) n + q^* n^{2/3},
\end{align*}
where $q^*$ is considered as the new parameter, and using the values in \eqref{eq:expansion_B} and \eqref{eq:expansions_LC}
\begin{align*}
	\alpha_0(1 - \beta_0) = 1/2.
\end{align*}
We see that $q = q^*/(1 - \beta_0) + O(n^{\varepsilon - 1/6})$ that is, we certainly have $\A(cq) \sim \A(cq^*/(1 - \beta_0))$.
Similarly, $m = \beta_0 t + r$ rewrites to
\begin{align*}
	m = \frac{\beta_0}{1 - \beta_0} t^* + r^*,
	\quad \mbox{where} \quad
	r^* = \frac{r}{1 - \beta_0}.
\end{align*}
In particular, this means that if $t^*$ is fixed and $m$ varies then corresponding consecutive $r$ differ by $1 - \beta_0$.
Hence, we obtain
\begin{align*}
	\frac{1}{\alpha_0} \frac{1}{[z^n]\, C(z)} \sum_{m,t:\, t-m = t^*} c_{n,t,m}
	& \sim \frac{1}{\alpha_0} \sum_{r^* = r/(1 - \beta_0) = O(n^{1/2 + \varepsilon})}
		\frac{B_3}{C_3}\frac{(1 + \tau L_0)^{5/2}}{\alpha_0^{3/2}}
		\frac{e^{-r^2/(2\sigma^2\alpha_0n)}}{\sqrt{2\pi\sigma^2\alpha_0n}} n^{-2/3} c\A(cq) \\
	& \sim \frac{1}{\alpha_0} \frac{B_3}{C_3} \frac{(1 + \tau L_0)^{5/2}}{\alpha_0^{3/2}}
		\frac{1}{1 - \beta_0} n^{-2/3} c\A(cq^*/(1 - \beta_0)) \\
	& = n^{-2/3} c^*\A(c^*q^*),
\end{align*}
where the last equality is implied by \eqref{eq:cst_one} and by setting $c^* = c/(1 - \beta_0) = 4/(1 - 1/\sqrt{3})^{4/3}$.

\medskip

Finally, we make a plausibility check and note that for any fixed constant $a > 0$ we have
\begin{equation*}
	n^{-2/3}\sum_{|x|\le K \, : \, x n^{2/3} \in \mathbb{Z}} a{\mathcal A}\left(ax\right) = 1 + o(1),
	\qquad\text{as}\quad K, n\to\infty.
\end{equation*}
Hence, all neglected parts in our computations have no asymptotic weight.
This concludes the proof.
\qed

\paragraph{Proof of Theorem \ref{thm:largest_3comp}.}

We only give a short sketch of the proof here, as it follows the lines of the proof of Theorem \ref{thm:largest_cubic_block}.
First, to obtain a decomposition of cubic maps in terms of their 3-cores, we consider the \textit{near 3-core} of a cubic map, that is, we contract each bead then each cherry to a single vertex of degree two.
The family of near 3-cores of cubic maps can be obtained by considering all 3-connected cubic maps then possibly replacing each of their edges by a path of two edges.
Special care must be taken to re-root the resulting map when the original root edge was effectively replaced.

Let $C(z,w,u)$ be the generating function counting cubic maps where the variable $z$ marks the number of edges, while $w$ marks the number of edges and $u$ the number of vertices of the near 3-core.
And let $M(y)$ be the generating function counting 3-connected cubic maps where $y$ marks the number of edges.
Then the decomposition of cubic maps in terms of their near 3-core gives an equation analogue to \eqref{eq:c^*core}:
\begin{equation}\label{eq:decomp_3core}
	C(z,w,u) = M\big( zw(1 + zwuD(z)) \big)\frac{1 + 2zwuD(z)}{1 + zwuD(z)} + zA(z),
\end{equation}
where both $D(z)$ and $A(z)$ encode cubic maps without their root edge, $D(z)$ count those not rooted at an isthmus, while $A(z)$ count those without a 3-core, namely
\begin{equation*}
	A(z) = L(z) + I(z) + P(z) + (D(z) - H(z))(D(z) - H(z) - S(z)).
\end{equation*}

Again, omitting periodicities and setting $z\to y^3$ in \eqref{puis:3conn_cubic}, we get that $M(y)$ is $\Delta$-analytic at $\theta = 2^{1/3}3/8$, and furthemore we have for $y\sim\theta$ that
\begin{equation}\label{eq:expansion_M_edges}
	M(y) = M_0 - M_2Y^2 + M_3Y^3 + O(Y^4), \qquad Y = \sqrt{1 - \frac{y}{\theta}},
\end{equation}
with $M_0 = 5/256$, $M_2 = 63/256$ and $M_3 = 3\sqrt{2}/8$.
Similarly, setting $z\to z^3$ in the proof of Theorem \ref{thm:enum}(a) implies that $D(z)$ is $\Delta$-analytic at $\sigma = 2^{1/3}\sqrt{3}/6$, and for $z\sim\sigma$
\begin{equation}\label{eq:expansion_D_edges}
	D(z) = D_0 - D_2Y^2 + D_3Y^3 + O(Z^4), \qquad Z = \sqrt{1 - \frac{z}{\sigma}},
\end{equation}
with $D_0 = 2^{2/3}(9/4 - \sqrt{3})$, $D_2 = 2^{2/3}(9/2 + \sqrt{3})$ and $D_3 = 2^{1/6}36$.
Then, we apply the change of variables $x = zw$ and $v = uD(z)$ to \eqref{eq:decomp_3core}, and focus on the composition scheme
\begin{equation*}
	M\big( x(1 + xv) \big)\frac{1 + 2xv}{1 + xv},
\end{equation*}
using the exact same strategy as in the proof of Theorem \ref{thm:largest_cubic_block} (including an analogue version of Lemma \ref{lem:double-counting_m}), but where $D_2$ plays the r\^ole of $L_2$, $D_3$ of $L_3$, and $2\theta(1 + 4\theta D_0)^{-1/2}(1 + \sqrt{1 + 4\theta D_0})^{-1}$ of $(\tau^{-1} + L_0)^{-1}$.
Thus, the constants $\alpha_0$, $\beta_0$ and $c$ become
\begin{align*}
	\alpha_0 = \frac{2\theta - \sigma}{2\theta - \sigma(1 - \sigma D_2)} = \frac{1}{2} - \frac{\sqrt 3}{9}, \quad
	\beta_0 = \frac{\sigma^2D_0}{2\theta - \sigma} = \frac{19}{46} - \frac{3\sqrt 3}{23}
	\quad\text{and}\quad
	c = \frac{1}{\alpha_0}\left(\frac{3D_3}{(1 - \alpha_0)D_2}\right)^{2/3}.
\end{align*}
And, for $q$ in a bounded interval, we obtain the final estimate
\begin{equation*}
	\mathbb{P}\left( Z_n = \lfloor \alpha_0(1 - \beta_0)n + qn^{2/3} \rfloor \right)
	\sim \frac{M_3}{C_3}\left(\frac{1 + 2\sigma D_0}{1 + \sigma D_0}\right)^{5/2}\alpha_0^{-5/2} n^{-2/3} c'\A(c'q)
	\qquad\text{as } n\to\infty.
\end{equation*}
From the values in \eqref{eq:expansions_LC}, \eqref{eq:expansion_M_edges} and \eqref{eq:expansion_D_edges}, we get $\alpha_0(1 - \beta_0) = 1/4$, $c' = c/(1 - \beta_0) = 72(3/2 - 1/\sqrt{3})^{-4/3}$ and
\begin{equation*}
	\frac{M_3}{C_3}\left(\frac{1 + 2\sigma D_0}{1 + \sigma D_0}\right)^{5/2}\alpha_0^{-5/2} = 1,
\end{equation*}
which concludes the proof.
\qed

%
%
\section{Concluding remarks}\label{sec:conclusions}

We include a table of small values for the various number sequences counting cubic planar maps that are 2-connected $(b_n)$, arbitrary $(c_n)$, 2-connected simple $(b_n^*)$, simple $(c_n^*)$, 2-connected triangle-free $(g_n)$, triangle-free $(f_n)$, 2-connected triangle-free simple $(g_n^*)$, triangle-free simple $(f_n^*)$.
The index $n$ is now the \textit{total} number of faces.
For completeness we also include the numbers $t_n=[z^n]\, M(z)$ of 3-connected cubic maps, which are equal to the numbers of 3-connected triangulations.

\begin{equation*}
	\begin{array}{rrrrrrrrrr}
		n & t_n & b_n & c_n & b_n^* & c_n^* & g_n & f_n & g_n^* & f_n^*\\
		\hline
		3 & & 1 & 4 & & & 1 & 4 \\
		4 & 1 & 4 & 32 & 1 & 1 & 3 & 19 \\
		5 & 3 & 24 & 336 & 3 & 3 & 12 & 147 \\
		6 & 13 & 176 & 4096 & 19 & 19 & 64 & 1432 & 1 & 1 \\
		7 & 68 & 1456 & 54912 &	128 & 143 & 432 & 16547 & 3 & 3 \\
		8 & 399 & 13056 & 786432 & 909 & 1089 & 3244 & 206520  & 12 & 12 \\
		9 & 2530 & 124032 & 11824384 & 6737 & 8564 & 2596 & 2707135 & 59 & 59 \\
		10 & 16965 & 1230592 & 184549376 & 51683 & 69075 & 217806 & 36818912 & 325 & 325 \\
		11 & 118668 & 12629760 & 2966845440 & 407802 & 569469 & 1893226 & 515736964 & 1863 & 1890  \\
		\hline
		\hbox{\it OEIS} & \hbox{\it A000260} & \hbox{\it A000309} & \hbox{\it A002005} & \hbox{\it A058860} & \hbox{\it A058859}
	\end{array}
\end{equation*}

To conclude, let us mention that it would be also possible to analyse the size of the largest block in random cubic planar graphs (see \cite{NRR20} for recent results on this topic).
The main differences with the present work are the following.
Given a vertex-rooted cubic planar graph $G$ one cannot just define the 2-core as the block containing the root vertex, as there may be several such blocks.
This can be circumvented by considering coreless graphs as in \cite{GNR13}.
However the main difficulty, even if the 2-core is well defined, is that it may contain double edges.
One needs thus to consider rooted 2-connected cubic planar graphs counted according to the number of vertices and double edges (the total number of edges is determined by the number of vertices).
If $B^\bullet(x,y)$ is the associated generating function, where $x$ marks vertices and $y$ marks double edges, then the composition scheme is
\begin{equation}\label{eq:core-graphs}
B^\bullet\left(wxQ(x,w)^{3/2},\frac{Q(x,w)^2-1}{2Q(x,w)^2}\right),
\end{equation}
where $w$ marks the size (number of vertices) of the 2-core and $Q(x,w)$ is the generating function of sequences of `cherries', playing the same role as $1/(1-zL(z))$ in Equation \eqref{eq:c^*core}.

We see that this is a \textit{bivariate} composition scheme in which the substitution in both variables contributes to the size of the core.
It is possible to show that a bivariate scheme of the form $C(wH(z), F(z))G(z)$ leads to a map-Airy law assuming suitable analytic conditions.
And we believe this can be extended to prove a map-Airy law for the size of the largest block in a random cubic planar graph (and also for a random simple cubic map), but this would be technically more demanding and we leave it as a future project.

\bibliographystyle{abbrv}
\bibliography{biblio_cubic_maps}

\end{document}